\documentclass{amsart}

\usepackage{amssymb, amsmath, amsthm, amsfonts}
\usepackage{mathrsfs,comment}
\usepackage{graphicx}
\usepackage{placeins}
\usepackage{rotating}
\usepackage{tikz}
\usepackage{float}
\usepackage{subfigure}
\usepackage{hvfloat}
\usepackage{caption}
\usepackage{pdflscape}
\usepackage{hyperref}  
\usepackage{url}
\usepackage[all,arc,2cell]{xy}
\UseAllTwocells
\usepackage{enumerate}
\usepackage{chngcntr}
 \usepackage{lineno}
 \usepackage{blindtext}
\usepackage{verbatim}
\usepackage{soul}
\usepackage[normalem]{ulem}
\usepackage{stmaryrd}
\usepackage{todonotes}

\usepackage{lscape}

\usepackage[makeroom]{cancel}

\hypersetup{%
  bookmarksnumbered=true,%
  bookmarks=true,%
  colorlinks=true,%
  linkcolor=blue,%
  citecolor=blue,%
  filecolor=blue,%
  menucolor=blue,%
  pagecolor=blue,%
  urlcolor=blue,%
  pdfnewwindow=true,%
  pdfstartview=FitBH}

\def\@url#1{{\tt\def~{\lower3.5pt\hbox{\char'176}}\def\_{\char'137}#1}}

\let\fullref\autoref
\def\makeautorefname#1#2{\expandafter\def\csname#1autorefname\endcsname{#2}}
\makeautorefname{equation}{Equation}%
\makeautorefname{footnote}{footnote}%
\makeautorefname{item}{item}%
\makeautorefname{figure}{Figure}%
\makeautorefname{table}{Table}%
\makeautorefname{part}{Part}%
\makeautorefname{appendix}{Appendix}%
\makeautorefname{chapter}{Chapter}%
\makeautorefname{section}{Section}%
\makeautorefname{subsection}{Section}%
\makeautorefname{subsubsection}{Section}%
\makeautorefname{paragraph}{Paragraph}%
\makeautorefname{subparagraph}{Paragraph}%
\makeautorefname{theorem}{Theorem}%
\makeautorefname{theo}{Theorem}%
\makeautorefname{thm}{Theorem}%
\makeautorefname{addendum}{Addendum}%
\makeautorefname{addend}{Addendum}%
\makeautorefname{add}{Addendum}%
\makeautorefname{maintheorem}{Main theorem}%
\makeautorefname{mainthm}{Main theorem}%
\makeautorefname{corollary}{Corollary}%
\makeautorefname{claim}{Claim}%
\makeautorefname{corol}{Corollary}%
\makeautorefname{coro}{Corollary}%
\makeautorefname{cor}{Corollary}%
\makeautorefname{lemma}{Lemma}%
\makeautorefname{lemm}{Lemma}%
\makeautorefname{lem}{Lemma}%
\makeautorefname{sublemma}{Sublemma}%
\makeautorefname{sublem}{Sublemma}%
\makeautorefname{subl}{Sublemma}%
\makeautorefname{proposition}{Proposition}%
\makeautorefname{proposit}{Proposition}%
\makeautorefname{propos}{Proposition}%
\makeautorefname{propo}{Proposition}%
\makeautorefname{prop}{Proposition}%
\makeautorefname{property}{Property}
\makeautorefname{proper}{Property}
\makeautorefname{scholium}{Scholium}%
\makeautorefname{step}{Step}%
\makeautorefname{conjecture}{Conjecture}%
\makeautorefname{conject}{Conjecture}%
\makeautorefname{conj}{Conjecture}%
\makeautorefname{question}{Question}
\makeautorefname{questn}{Question}
\makeautorefname{quest}{Question}
\makeautorefname{ques}{Question}
\makeautorefname{qn}{Question}
\makeautorefname{definition}{Definition}%
\makeautorefname{defin}{Definition}%
\makeautorefname{defi}{Definition}%
\makeautorefname{def}{Definition}%
\makeautorefname{defn}{Definition}%
\makeautorefname{dfn}{Definition}%
\makeautorefname{notation}{Notation}
\makeautorefname{nota}{Notation}
\makeautorefname{notn}{Notation}
\makeautorefname{remark}{Remark}%
\makeautorefname{rema}{Remark}%
\makeautorefname{rem}{Remark}%
\makeautorefname{rmk}{Remark}%
\makeautorefname{rk}{Remark}%
\makeautorefname{remarks}{Remarks}%
\makeautorefname{rems}{Remarks}%
\makeautorefname{rmks}{Remarks}%
\makeautorefname{rks}{Remarks}%
\makeautorefname{example}{Example}%
\makeautorefname{examp}{Example}%
\makeautorefname{exmp}{Example}%
\makeautorefname{exmps}{Examples}%
\makeautorefname{exam}{Example}%
\makeautorefname{exa}{Example}%
\makeautorefname{algorithm}{Algorith}%
\makeautorefname{algo}{Algorith}%
\makeautorefname{alg}{Algorith}%
\makeautorefname{axiom}{Axiom}%
\makeautorefname{axi}{Axiom}%
\makeautorefname{ax}{Axiom}%
\makeautorefname{case}{Case}%
\makeautorefname{claim}{Claim}%
\makeautorefname{clm}{Claim}%
\makeautorefname{assumption}{Assumption}%
\makeautorefname{assumpt}{Assumption}%
\makeautorefname{conclusion}{Conclusion}%
\makeautorefname{concl}{Conclusion}%
\makeautorefname{conc}{Conclusion}%
\makeautorefname{condition}{Condition}%
\makeautorefname{condit}{Condition}%
\makeautorefname{cond}{Condition}%
\makeautorefname{construction}{Construction}%
\makeautorefname{construct}{Construction}%
\makeautorefname{const}{Construction}%
\makeautorefname{cons}{Construction}%
\makeautorefname{criterion}{Criterion}%
\makeautorefname{criter}{Criterion}%
\makeautorefname{crit}{Criterion}%
\makeautorefname{exercise}{Exercise}%
\makeautorefname{exer}{Exercise}%
\makeautorefname{exe}{Exercise}%
\makeautorefname{problem}{Problem}%
\makeautorefname{problm}{Problem}%
\makeautorefname{probm}{Problem}%
\makeautorefname{prob}{Problem}%
\makeautorefname{solution}{Solution}%
\makeautorefname{soln}{Solution}%
\makeautorefname{sol}{Solution}%
\makeautorefname{summary}{Summary}%
\makeautorefname{summ}{Summary}%
\makeautorefname{sum}{Summary}%
\makeautorefname{operation}{Operation}%
\makeautorefname{oper}{Operation}%
\makeautorefname{observation}{Observation}%
\makeautorefname{observn}{Observation}%
\makeautorefname{obser}{Observation}%
\makeautorefname{obs}{Observation}%
\makeautorefname{ob}{Observation}%
\makeautorefname{convention}{Convention}%
\makeautorefname{convent}{Convention}%
\makeautorefname{conv}{Convention}%
\makeautorefname{cvn}{Convention}%
\makeautorefname{warning}{Warning}%
\makeautorefname{warn}{Warning}%
\makeautorefname{note}{Note}%
\makeautorefname{fact}{Fact}%

  \makeatletter
                   \let\c@lemma\c@theorem
                  \makeatother

\newtheorem{thm}{Theorem}[section]

\newtheorem{cor}{Corollary}[section]
\newtheorem{prop}{Proposition}[section]
\newtheorem{lem}{Lemma}[section]

\theoremstyle{definition}

\newtheorem{rem}{Remark}[section]

\newtheorem{notation}{Notation}[section]

\makeatletter
\let\c@lem=\c@thm
\let\c@cor=\c@thm
\let\c@prop=\c@thm
\let\c@lem=\c@thm
\let\c@defn=\c@thm
\let\c@exmps=\c@thm
\let\c@rem=\c@thm
\let\c@warn=\c@thm
\let\c@claim=\c@thm
\let\c@quest=\c@thm
\let\c@notation=\c@thm
\let\c@note=\c@thm
\makeatother

\numberwithin{equation}{section}

\newcommand{\Z}{\mathbb{Z}}

\newcommand{\Q}{\mathbb{Q}}

\newcommand{\F}{\mathbb{F}}

\newcommand{\G}{\mathbb{G}}
\newcommand{\W}{\mathbb{W}}

\newcommand{\cE}{\mathcal{E}}

\DeclareSymbolFontAlphabet{\scr}{rsfs}

\newcommand{\smsh}{\wedge}

\newcommand{\xra}{\xrightarrow}

\def\quickop#1{\expandafter\newcommand\csname #1\endcsname{\operatorname{#1}}}
\quickop{Hom} \quickop{End} \quickop{Aut} \quickop{Tel} \quickop{Mic} 
\quickop{Ext} \quickop{Tor} \quickop{Cotor} \quickop{Id} \quickop{Coker} \quickop{Ker}
\quickop{Lim} \quickop{Colim} \quickop{Holim} \quickop{Hocolim}
\quickop{id} \quickop{tel} \quickop{mic} \quickop{coker} 
\quickop{colim} \quickop{holim} \quickop{hocolim} \quickop{im}
\DeclareMathOperator{\Gal}{Gal}

\DeclareMathOperator{\Pic}{Pic}

\newcommand{\mono}{{\mathrm{M}}} %monochromatic layer M_n
\newcommand{\Kn}{{\mathrm{K}}} %Morava K-theory K(n)
\newcommand{\E}{{\mathrm{E}}} 
\newcommand{\Dn}{{\mathrm{D}}} 
\newcommand{\I}{{\mathrm{I}}} % Morava E-theory E_m
\newcommand{\ssE}{{\mathscr{E}}}%E for ss 
\newcommand{\ssM}{{\mathscr{M}}}%E for ss 
\newcommand{\ssF}{{\mathscr{F}}}%F for ss 
\newcommand{\ssK}{{\mathscr{K}}}%K for ss 

\newcommand{\Sdet}{S\langle {\det} \rangle}
\newcommand{\Edet}{\E_*\langle {\det} \rangle}

\definecolor{darkspringgreen}{rgb}{0.09, 0.45, 0.27}
\definecolor{darkterracotta}{rgb}{0.8, 0.31, 0.36}
	\definecolor{darkcoral}{rgb}{0.8, 0.36, 0.27}
	\definecolor{indiagreen}{rgb}{0.07, 0.53, 0.03}
	\definecolor{mountainmeadow}{rgb}{0.19, 0.73, 0.56}
	\definecolor{mountbattenpink}{rgb}{0.6, 0.48, 0.55}
	\definecolor{palatinatepurple}{rgb}{0.41, 0.16, 0.38}
	\definecolor{cinnamon}{rgb}{0.82, 0.41, 0.12}
	\definecolor{chocolate}{rgb}{0.82, 0.41, 0.12}
	
\definecolor{darkspringgreen}{rgb}{0.09, 0.45, 0.27}
\newcommand{\stkout}[1]{\ifmmode\text{\sout{\ensuremath{#1}}}\else\sout{#1}\fi}

\DeclareMathOperator{\fib}{fib}

\usepackage{pdflscape}

\title{Gross--Hopkins Duals of Higher Real K--theory Spectra}

\author[Barthel]{Tobias Barthel}
\address[Barthel]{Max Planck Institute for Mathematics}
\email{tbarthel@mpim-bonn.mpg.de}
\author[Beaudry]{Agn\`es Beaudry}
\address[Beaudry]{Department of Mathematics, University of Colorado Boulder}
\email{agnes.beaudry@colorado.edu}
\author[Stojanoska]{Vesna Stojanoska}
\address[Stojanoska]{Department of Mathematics, University of Illinois at Urbana-Champaign}
\email{vesna@illinois.edu}

\thanks{This material is based upon work supported by the National Science Foundation under Grant No.~DMS-1606479 and Grant No.~DMS-1612020/1725563. The first-named author was partially supported by the DNRF92.}


\begin{document}

\maketitle

\begin{abstract}
We determine the Gross--Hopkins duals of certain higher real $K$--theory spectra. More specifically, let $p$ be an odd prime, and consider the Morava $E$--theory spectrum of height $n=p-1$. It is known, in the expert circles, that for certain finite subgroups $G$ of the Morava stabilizer group, the homotopy fixed point spectra $ E_n^{hG}$ are Gross--Hopkins self-dual up to a shift. In this paper, we determine the shift for those finite subgroups $G$ which contain $p$--torsion. This generalizes previous results for $n=2$ and $p=3$.
\end{abstract}

%\setcounter{tocdepth}{1}
%\tableofcontents

\section{Introduction}\label{sec:state}

To understand the sphere spectrum at a prime $p$, one first studies its building blocks, which mirror the stratification of formal groups by height; these are the $K(n)$--local spheres $S_{K(n)}$. The behavior at height $n$ is governed by a profinite group of virtual cohomological dimension $n^2$, the Morava stabilizer group $\G_n$. One of the most stunning theorems of chromatic homotopy theory states that $S_{K(n)}$ is the homotopy fixed point spectrum for the action of $\G_n$ on Morava $E$--theory $E_n$. This is a complex oriented ring spectrum whose formal group law is a universal deformation of the formal group law of $K(n)$, a height $n$ formal group law. Computing the homotopy groups of $K(n)$--local spectra is thus intimately related to computing continuous group cohomology for $\G_n$ via homotopy fixed point spectral sequences.

When $p$ is large with respect to $n$, these spectral sequences collapse and the problems become entirely algebraic, although notoriously difficult if $n>1$. The case of interest in this paper is $n=p-1$. This is the tipping point for the difficulty in these types of computations as it is just outside of the algebraic range. At these heights, lying between Morava $E$--theory and the $K(n)$--local sphere, are the so-called ``higher real $K$--theories''. The name is justified, as for $n=1$ and $p=2$ one gets real $K$--theory. These spectra are constructed as homotopy fixed points $E_n^{hH}$ of Morava $E$--theory with respect to certain finite subgroups $H$ of the Morava stabilizer group. Henn \cite{henn_res} has constructed algebraic resolutions that indicate that the $K(n)$--local sphere can be realized as the inverse limit of a finite tower of fibrations built using higher real $K$--theory spectra. These are topologically realized at large primes in the same reference; more work is needed at small primes. (For $n=1$ and $p=2$ this is classical, see \cite{HMS_pic}; for $n=2$ and $p=3$, it is the subject of \cite{ghmr}, the case $n=p-1$ is treated in \cite{henn_res}, while $n=2=p$ is accomplished in \cite{goerss_bobkova}.) In any case, the expectation is that the resolutions should be realizable for any $n$ and $p$. Understanding the higher real $K$--theory spectra and their module categories is thus a first step towards understanding $K(n)$--local spectra at $n=p-1$. Furthermore, the higher real $K$--theories have the advantage of being computationally tractable. For example, Hopkins and Miller computed the homotopy groups of higher real $K$--theory spectra at $n=p-1$, modulo the image of a transfer map. 

The main result of this paper, \fullref{main:intro}, fits into a broader program to study duality and its implications for the Picard group of the $K(n)$--local category at $n=p-1$. The algebro-geometric manifestation of the duality in question is Grothendieck--Serre duality on Lubin--Tate space, due to Gross and Hopkins~\cite{grosshopkins2,grosshopkins1}. In full, the duality takes into account the action by $\G_n$; here we restrict to finite subgroups and study Gross--Hopkins duality for higher real $K$--theory spectra. 

To give context to our result, we need to introduce some notation. Since we work in the $K(n)$--local category, we will write $X \smsh Y$ for the $K(n)$--local smash product $(X \smsh Y)_{K(n)}$ and $(E_n)_*X$ for $\pi_*(X \smsh E_n)_{K(n)}$. As usual, $\Pic_n$ denotes the Picard group of the category of $K(n)$--local spectra, so it consists of ($K(n)$--local) equivalence classes of spectra $X$ for which there exists $Y$ with
$X \smsh Y \simeq S_{K(n)}$. The exotic Picard group $\kappa_n$ is the subgroup of $\Pic_n$ consisting of those elements $X \in \Pic_n$ such that $(E_n)_*X \cong (E_n)_*$ as $\G_n$--equivariant $(E_n)_*$--modules, i.e., as Morava modules. 

We propose to use higher real $K$--theory spectra to detect exotic elements at $n=p-1$ as follows. Let $\Pic(E_n^{hH})$ denote the Picard group of the category of $K(n)$--local $E_n^{hH}$--module spectra. For any subgroup $H \subseteq \G_n$, there is a map $\Pic_n \to \Pic(E_n^{hH})$ defined by $X \mapsto X \smsh E_n^{hH}$. One may be able to test the non-triviality of $X \in \kappa_n$ by showing that $X \smsh E_n^{hH}$ is non-trivial in $\Pic(E_n^{hH})$. In contrast with $\Pic_n$, the Picard group of $E_n^{hH}$ is known and simple;
Heard, Mathew, and Stojanoska \cite{matstohea_piceo} have shown it is cyclic, so that $X \smsh E_n^{hH} \simeq \Sigma^{k} E_n^{hH}$ for some integer $k$. The spectra $E_n^{hH}$ are periodic, so the problem reduces to determining whether or not $k\equiv 0$ modulo the periodicity.

Turning to Gross--Hopkins duality, let $I_{\Q/\Z}$ be the Brown--Comenetz spectrum, defined as the representing spectrum for the cohomology theory which assigns to
$X$ the abelian group $\Hom(\pi_{0}(X),\Q/\Z)$.
For $X$ a $K(n)$--local spectrum, the Gross--Hopkins dual of $X$ is defined as $I_nX  = F(M_nX, I_{\Q/\Z})$, where $M_nX$ is the monochromatic layer of $X$. Note that by a result of Greenlees and Sadofsky~\cite{greensad_tate}, if $H$ is a finite subgroup of $\G_n$, the Tate construction $E_n^{tH}$ vanishes $K(n)$--locally, so the norm $(E_n)_{hH} \to (E_n)^{hH}$ is a $K(n)$--equivalence. It follows that
$I_n(E_n^{hH}) \simeq I_n((E_n)_{hH}) \simeq (I_nE_n)^{hH}$, and similarly, for the Spanier--Whitehead duals $D(E_n^{hH}) \simeq (DE_n)^{hH}$. Therefore, we may write $I_nE_n^{hH}$ or $DE_n^{hH}$ unambiguously. At height $n=p-1$, it seems to be well-known that $I_nE_n^{hH}$ is an invertible $E_n^{hH}$--module; we give a detailed proof in \fullref{prop:diffsforced}. 
It follows by the remarks above that $I_nE_n^{hH} \simeq \Sigma^{k_I} E_n^{hH}$, for some integer $k_I$.

The main purpose of this paper is to determine the shift $k_I$ for certain finite subgroups of $\G_n$. Let $\mathbb{S}_n$ be the small Morava stabilizer group, that is, the automorphism group of the Honda formal group law over $\F_{p^n}$. At height $n=p-1$, $\mathbb{S}_n$ contains a maximal finite subgroup $F \cong C_p \rtimes C_{n^2}$. It can be extended by the Galois group to give a group $G$ which is a maximal finite subgroup of $\G_n$. For example, at $p=3$, $G$ is the group denoted by $G_{24}$ in \cite{ghmr}. We can now state our main result, which is proved as Theorems~\ref{thm:main} and \ref{thm:mainF} in \fullref{sec:final}.
\begin{thm}\label{main:intro}
For $p\geq 3$, $n=p-1$, there are equivalences $I_n E_n^{hC_p} \simeq \Sigma^{n^2} E_n^{hC_p}$,  $I_n E_n^{hF} \simeq \Sigma^{np^2+n^2 } E_n^{hF}$ and $I_nE_n^{hG} \simeq \Sigma^{np^2+n^2 } E_n^{hG}$.
\end{thm}

This result is well-known when $p=3$ for the group $G$: it follows from work of Mahowald--Rezk \cite{MahowaldRezk}, as described in Behrens \cite[Proposition 2.4.1]{behrens_mod}. Alternatively, it also follows by $K(2)$--localizing the result of \cite{Stoj-th}, and in some sense, our proof (for any $p$) specializes to a $K(2)$--local variant of that paper.
However, our computational techniques are different. The input to \cite{MahowaldRezk} is the mod--$p$ homology of connective versions of the spectra of interest and
 \cite{Stoj-th} relies heavily on algebro-geometric input. In contrast, the methods used in this paper are intrinsically $K(n)$--local. 

Of course, the theorem also has a well-known analogue at height $n=1$ and $p=2$. Namely, in this case, $E_1$ is $2$--completed complex $K$--theory, the Galois group is trivial and 
$C_2$ is the unique non-trivial finite subgroup acting on $E_1$ by complex conjugation with homotopy fixed points being $2$--completed real $K$--theory. By $K(1)$--localizing the result of Anderson \cite{Anderson} (see \cite{HS-KRD} for a more modern approach), one gets $I_1(E_1^{hC_2}) \simeq \Sigma^5 E_1^{hC_2}$. 

For $H=C_p$, $F$ or $G$ in \fullref{main:intro}, we prove the result by comparing the homotopy fixed point spectral sequence of $I_n(E_n^{hH})\simeq(I_nE_n)^{hH}$ to the $\Q/\Z$--dual of the homotopy orbit spectral sequence of $(M_n E_n)_{hH}$. The key computational ingredient is then a theorem of Hopkins--Miller which provides an explicit description of the $H$--module structure of $(E_n)_*$ at height $n=p-1$. To finish this introduction, we give an indication of how our result can be used to detect a family of interesting invertible $K(n)$--local spectra.

The Gross--Hopkins dual of the sphere, $I_n= I_nS_{K(n)}$, is of particular interest as a dualizing object. The spectrum $I_n$ is an element of $\Pic_n$ by the work of Hopkins and Gross \cite{grosshopkins1,grosshopkins2}; see also ~\cite[Theorem 10.2(e)]{hovstrmemoir} and \cite[Theorem 2]{StrickGrossHop}. In fact, their work implies an equivalence 
\begin{equation}\label{eqn:In} I_n \simeq \Sdet \smsh S^{n^2-n} \smsh P_n,\end{equation}
where $P_n$ is an element of $\kappa_n$ and $\Sdet \in \Pic_n$ can be described as follows.

Following \cite{det}, let $S(1)$ be the $p$-complete
sphere with its natural action of $\Z_p^{\times} \cong (\pi_0S_p)^{\times}$. 
Then the determinant sphere is defined as
\[\Sdet :=  (E_n \smsh S(1))^{h\G_n}.\]
In \cite{det}, it is shown that the spectrum $\Sdet$ satisfies
$(E_n)_*\Sdet \cong (E_n)_*$ as $(E_n)_*$--module, but its action of $\G_n$ is twisted by the determinant (see \eqref{eqn:actdet}). Moreover, a key property of $\Sdet$ proved in \cite{det} is that, for any closed subgroup $H$ of the stabilizer group $\mathbb{G}_n$ which is in the kernel of the determinant, $E_n^{hH}\wedge S \langle \det \rangle\simeq E_n^{hH}$.

At large primes, $\kappa_n$ is trivial (see \cite[Proposition~7.5]{HMS_pic}) so that $I_n$ is equivalent to $ \Sdet\smsh S^{n^2-n} $. 
However, in cases when $\kappa_n$ is non-trivial, $P_n$ can provide an interesting twist. This happens at $n=2$ and $p=3$; Goerss and Henn in \cite{GH_bcdual} show that $P_2$ has order 3. 
In particular, they show that $P_2$ is detected by the higher real $K$--theory spectrum $E_2^{hG}$. This raises the natural question of whether or not higher real $K$--theory spectra detect $P_n$ at height $n=p-1$. 

A strategy to answer the question more generally is to combine our result with information about the $K(n)$--local Spanier--Whitehead dual $DE_n^{hH}= F(E_n^{hH}, S_{K(n)})$. Again, $DE_n^{hH}$ is an invertible $E_n^{hH}$--module (see \fullref{prop:diffsforced}) so that $DE_n^{hH} \simeq \Sigma^{k_D} E_n^{hH}$, for some integer $k_D$. Further, because $I_n$ is invertible, we have $I_n E_n^{hH} \simeq I_n \smsh DE_n^{hH}$.

If we assume further that $H$ is in the kernel of the determinant (for example, if $H=C_p$), then $E_n^{hH} \smsh \Sdet \simeq E_n^{hH}$. 
Then \eqref{eqn:In} implies that
\begin{equation*}
 P_n \smsh E_n^{hH} \simeq \Sigma^{k_I-k_D+n-n^2}E_n^{hH}.   
\end{equation*}
Provided that $k_I-k_D+n-n^2$  is not congruent to the periodicity of $E_n^{hH}$, the higher real $K$--theory spectrum $E_n^{hH}$ will detect $P_n$ as non-trivial. Here we determine $k_I$; to gain insight on whether $H$ detects $P_n$ in general, it remains to determine the more elusive shift $k_D$.

\subsection*{Notation and conventions}

Throughout this article, we fix an odd prime $p$, and we work at height $n = p-1$. 
Let $\Kn=K(n)$ be Morava $K$--theory chosen so that the formal group law of $\Kn$ is the Honda formal group law. Let $\E=E_n$ be Morava $E$--theory with coefficients
\[ \E_* =\W\llbracket u_1,\dots, u_{n-1}\rrbracket[u^{\pm 1}], \]
where $\W$ is the ring of Witt vectors of $\F_{p^n}$.
Let $\mathbb{S}=\mathbb{S}_n$ be the automorphism group of the formal group law of $\Kn$ and $\G =\G_n$ be its extension by the Galois group of $\F_{p^n}/\F_p $. Often these are called the small and big Morava stabilizer group, respectively. 

For $(-)_{A}$ the Bousfield localization at a spectrum $A$, recall that there is a natural transformation $(-)_{E_n} \to (-)_{E_{n-1}}$ whose fiber, denoted $\mono=M_n$, is called the $n$--th monochromatic layer. Unless otherwise specified, $S = S_{\Kn}$ is the $\Kn$--local sphere spectrum and, for $\Kn$--local spectra $X$ and $Y$, $X \smsh Y$ denotes $ (X \smsh Y)_{\Kn}$ and $\Dn(X) = F(X, S_{\Kn})$.

We let $D_{\Q/\Z}(-) = \Hom_{\Z}(-,\Q/\Z)$ denote Pontryagin duality, and $I_{\Q/\Z}$ is the Brown--Comenetz spectrum which represents the cohomology theory
\[ X \mapsto D_{\Q/\Z}(\pi_{*}(X)).\]
Note that for $ I_{\Q/\Z} X = F(X, I_{\Q/\Z})$, we have $\pi_{q} I_{\Q/\Z}  X  \cong D_{\Q/\Z}(\pi_{-q}(X))$.
For $X$ a $\Kn$--local spectrum, we let $\I X = I_nX$  be the Gross--Hopkins dual, defined as $F(\mono X, I_{\Q/\Z})$ and we let $\I = I_nS$.

We provide a little more information about the finite subgroups of $\G$. As was noted above, at height $n=p-1$, the group $\mathbb{S}$ contains a maximal finite subgroup $F \cong C_p \rtimes C_{n^2}$. We can give $F$ the presentation
\begin{equation}\label{eqn:F}
F = \left< \zeta, \tau \mid \zeta^p=1, \tau^{n^2}=1, \tau^{-1}\zeta \tau = \zeta^e\right>,
\end{equation}
where $e$ is a generator of $\Z/p^{\times}$. The group $F$ can be extended by the Galois group
\begin{equation}\label{eqn:G}
1 \to F \to G \to \Gal(\F_{p^n}/\F_p) \to 1,
\end{equation}
making $G$ a maximal finite subgroup of $\G$ as explained in \cite{henn_res} or \cite[Section 2]{heard_eop}. 
For example, at $p=3$, $G$ is the group denoted by $G_{24}$ in \cite{ghmr}.

\subsection*{Organization of the paper}

In \fullref{sec:prelim}, we explain how Pontryagin duality interacts with the homotopy fixed point and homotopy orbit spectral sequences. We then record a version of the geometric boundary theorem, which we apply in \fullref{sec:mono} to obtain a description of the Tate spectral sequence for the monochromatic layer, which is used in \fullref{sec:final} to prove our main results. 

\subsection*{Acknowledgements}
The authors would like to thank Mark Behrens, Paul Goerss, Hans-Werner Henn, Mike Hopkins and Craig Westerland for useful conversations, as well as the referee for useful suggestions and corrections.

\section{Preliminaries on the homotopy fixed point spectral sequence}\label{sec:prelim}

We recall the construction of the homotopy fixed point and homotopy orbit spectral sequences and explain how they are related via Pontryagin duality. 
Furthermore, we give a weak version of the geometric boundary theorem for the homotopy orbit spectral sequence for a finite group, which will be used in the next section.

\subsection{Dualizing the homotopy orbits spectral sequence}

We describe a spectral sequence, which, in certain cases, computes the homotopy groups of $I_{\Q/\Z}(X^{hH})$ where $H$ is a finite group and $X$ is an $H$--spectrum.

Let $H$ be a finite group and let $EH^{(s)}$ be the $s$'th skeleton of $EH$. Consider the family of cofiber sequences 
\begin{equation}\label{eqn:egskel}
\Sigma^{\infty}EH^{(s)}_{+} \to \Sigma^{\infty}EH^{(s+1)}_{+} \to \Sigma^{\infty}EH^{(s+1)}/EH^{(s)} \end{equation}
Applying $-\smsh_{H}X$ and taking homotopy groups gives rise to an exact couple whose associated spectral sequence is the homotopy orbits spectral sequence (HOSS), with 
 \[\ssE^2_{s,t} = H_s(H,\pi_tX) \Longrightarrow \pi_{t+s} X_{hH}\]
 and differentials $d_r^{\ssE}\colon \ssE^r_{s,t} \to \ssE^r_{s-r,t+r-1}$.
 
 Recall the definition of $I_{\Q/\Z}$ from \fullref{sec:state}. Further, note that
 \[ (I_{\Q/\Z}X)^{hH} \simeq I_{\Q/\Z}(X_{hH}).\]
Applying the functor $F_H(-, F(X, I_{\Q/\Z}))$ to \eqref{eqn:egskel} and taking homotopy groups gives the homotopy fixed point spectral sequence (HFPSS)
  \[\ssF_2^{s,t} = H^s(H,\pi_tI_{\Q/\Z}X) \Longrightarrow \pi_{t-s} ((I_{\Q/\Z}X)^{hH})\]
 with differentials $d_r^{\ssF}\colon \ssF_r^{s,t} \to \ssF_r^{s+r,t+r-1}$.
 However, note that 
 \[F_H(-, F(X, I_{\Q/\Z})) \simeq F(-\smsh_{H}X, I_{\Q/\Z}). \]
 Therefore, the HFPSS is isomorphic to the one associated to the exact couple obtained by applying $\pi_*F(-\smsh_{H}X, I_{\Q/\Z})$ to the cofiber sequences \eqref{eqn:egskel}, which is easily seen to be of signature
  \[\ssK_2^{s,t}   = D_{\Q/\Z}(H_s(H,\pi_{-t}X)) \Longrightarrow 
  \pi_{t-s}I_{\Q/\Z} (X_{hH}) \]
with differentials $d_r^{\ssK}\colon \ssK_r^{s,t} \to \ssK_r^{s+r,t+r-1}$ given by 
 \[ D_{\Q/\Z}(\ssE^r_{s,-t}) \xra{ D_{\Q/\Z}(d_r^{\ssE})}  D_{\Q/\Z}(\ssE^r_{s+r, -t-r+1}).\]
So the differentials $d_r^{\ssK}$ are completely determined by those in the HOSS. We record this in the following result.
 
 \begin{prop}\label{prop:dualhfss}
The homotopy fixed point spectral sequence
 \[ \ssF_2^{s,t} = H^{s}(H, \pi_t I_{\Q/\Z}X) \Longrightarrow \pi_{t-s} (I_{\Q/\Z}X)^{hH}  \]
 with differentials $d_r^{\ssF}\colon \ssF_r^{s,t} \to \ssF_r^{s+r,t+r-1}$ is Pontryagin dual to the homotopy orbits spectral sequence computing $\pi_*X_{hH}$. In particular, its differentials are completely determined by the latter.
 \end{prop}
 
Therefore, to compute $\pi_{*}(I_{\Q/\Z} X)^{hH}$, it suffices to fully understand the HOSS for $X_{hH}$. 

\begin{rem}\label{rem:tate}
Let $\widehat{H}^*(H, \pi_tX)$ denote the Tate cohomology of a finite group $H$ with coefficients in $\pi_tX$. There are maps of spectral sequences
\[ \xymatrix{ 
H^s(H, \pi_tX) \ar[r]  \ar@{=>}[d] &  \widehat{H}^{s} (H, \pi_tX) \ar[r]   \ar@{=>}[d] &  H_{-s-1}(H, \pi_tX)   \ar@{=>}[d] \\ 
\pi_{t-s}X^{hH} \ar[r] & \pi_{t-s}X^{tH} \ar[r] & \pi_{t-s-1}X_{hH}
} \]
such that the first map is multiplicative (see the discussion in \cite[Part II]{GreenleesMay}, in particular, (9.9) and Theorem 10.3). 
The maps in the top row of the diagram (i.e., the maps between $E_2$--pages) induce isomorphisms
\begin{align*}
\widehat{H}^{s} (H, \pi_tX) \cong  \begin{cases} H^{s} (H, \pi_tX) & s\geq 1 \\
H_{-s-1} (H, \pi_tX) & s \leq  -2.\\
\end{cases}
\end{align*} 
The map of spectral sequences
\[H^*(H, \pi_tX) \to \widehat{H}^*(H, \pi_tX)\]
is thus an isomorphism on $E_2$--pages for $*>0$ and there is an exact sequence
\[ 0 \to \widehat{H}^{-1}(H,\pi_tX) \to H_0(H, \pi_tX) \xra{N} H^0(H,  \pi_tX) \to \widehat{H}^0(H, \pi_tX) \to 0 \]
where $N$ is the algebraic norm map. For degree reasons, the image of $N$ consists of permanent cycles, so the differentials in the homotopy fixed point spectral sequence (HFPSS) are completely determined by those in the Tate spectral sequence (TSS). Further, at least in the case of groups with periodic cohomology if not more generally, the differentials in the HFPSS force all the differentials in the TSS, which in turn determines the differentials in the homotopy orbit spectral sequence (HOSS).

Depicting this in the $(t-s,s)$ plane, the HFPSS is in the first two quadrants, the TSS in all four, and the HOSS in the last two quadrants, and interference between the HFPSS and HOSS happens only along the horizontal axis, via the algebraic norm map.
\end{rem}

\subsection{Geometric Boundary}
We require an analogue of the geometric boundary theorem for the homotopy fixed point spectral sequence of a finite group. Such results are often stated in the context of generalized Adams spectral sequences (see \cite[Proposition 2.3.4]{ravgreen}). However, the most general statement in this vein that appears in the literature is \cite[Appendix A]{behrens_EHP} and \fullref{lem:gbt} is an application of Behrens's treatment.

\begin{lem}\label{lem:gbt}
Let $H$ be a finite group and suppose that $\Sigma^{-1}\mathcal{Z} \xra{j} \mathcal{X} \xra{i} \mathcal{Y} \xra{p} \mathcal{Z} $ is a fiber sequence of $H$--equivariant spectra which induces a short exact sequence
\begin{equation*}
\xymatrix{0 \ar[r] & \pi_*\mathcal{X} \ar[r] & \pi_*\mathcal{Y} \ar[r] & \pi_*\mathcal{Z} \ar[r] & 0}
\end{equation*}
on homotopy groups. Let $\ssE_r^{s,t}(\mathcal{X})$, $\ssE_r^{s,t}(\mathcal{Y})$, and $\ssE_r^{s,t}(\mathcal{Z})$ be the associated HFPSSs.
Suppose that $H^s(H,\pi_*\mathcal{Y}) = 0$ for $s>0$. Then there are inductively defined maps
\[
\xymatrix{\delta_r\colon \ssE_r^{s,t}(\mathcal{Z}) \ar[r] & \ssE_r^{s+1,t}(\mathcal{X})}
\]
induced by the connecting homomorphism $\delta_2 \colon H^s(H, \pi_t\mathcal{Z}) \to H^{s+1}(H, \pi_t\mathcal{X})$, such that 
\[d_r\delta_r = -\delta_rd_r.\] 
The maps $\delta_r$ are isomorphisms if $s>0$ and $\delta_{\infty}$ is a filtered version of $ j_* \colon \pi_*\mathcal{Z}^{hH} \to \pi_{*-1}\mathcal{X}^{hH}$. 
\end{lem}

\begin{rem}
Similar statements hold for the TSS and, respectively, the HOSS, provided that $\widehat{H}^s(H,\pi_*\mathcal{Y}) = 0$ for all $s$ and, respectively, that ${H}_s(H,\pi_*\mathcal{Y}) = 0$ for all $s<0$.
\end{rem}

The proof of \fullref{lem:gbt} is technical. It relies on a careful reading of \cite[Appendix A]{behrens_EHP} and, in particular, of the proof of Lemma A.4.1 in loc.cit. Understanding the remainder of the paper does not necessitate the understanding of the details of this particular proof, so the reader unenthusiastic about gory spectral sequence details can safely skip to the beginning of the next section if they so prefer.
\begin{proof}[Proof of \fullref{lem:gbt}]
In this argument, we assume that the reader is familiar with the notation and techniques of \cite[Appendix A]{behrens_EHP}. Let $X=\mathcal{X}^{hH}$ and
\[X_{\mu} =\begin{cases}
 F_H(EH^{(-\mu)}_+, \mathcal{X}) &  \mu\leq 0 \\
 \ast & \mu >0.
 \end{cases} \]
 Then $X_{\mu}$ is a tower under $X$ and
 \[X_{\mu}^{\mu} =  F_H(EH^{(-\mu)}_+/EH^{(-(\mu+1))}_+, \mathcal{X}) = F_{-\mu}. \]
 The spectral sequence of  \cite[Appendix A.3]{behrens_EHP} is a reindexing of the HFPSS
 \[E^r_{t, \mu}(X) \cong \ssE^{-\mu, t-\mu}_r(X) \]
 with differentials
 \[d_{\alpha}^{{X}} \colon E^\alpha_{t, \mu}(X)  \to E^\alpha_{t-1, \mu-\alpha}(X).    \]
We use the same notation for $Y = \mathcal{Y}^{hH}$ and $Z = \mathcal{Z}^{hH}$.

In particular, for any $X$ we have that $E^\alpha_{t, \mu}(X)=0$ if $\mu>0$. By our assumptions on $Y$, $E^2_{t,\mu}(Y) =0 $ for any $\mu \neq 0$ and 
there are exact sequences
\begin{align}\label{eqn:se}
 0 \to E_{t,\mu}^{1}(X) \xra{i_*}   E_{t,\mu}^1(Y)   \xra{p_*}   E_{t,\mu}^{1}(Z) \to 0.
 \end{align}
That is, $j_* \colon E_{t+1,\mu}^{1}(Z)  \to E_{t,\mu}^{1}(X) $ is the zero map.
We let
\[\delta_2\colon E_{t,\mu}^{2}(Z) \to E_{t-1,\mu-1}^{2}(X) \]
be the associated connecting homomorphism, which is identified with the usual connecting homomorphism in group cohomology. Assume by induction that $\delta_{\beta}$ has been defined on the $E^{\beta}$--term and is the map induced by $\delta_2$.
Suppose that there is a non-trivial differential $d_{\beta}^{Z}(w)=z$
for $w \in E_{t,\mu}^{\beta}(Z)$ and $z \in E_{t-1, \mu-\beta}^{\beta}(Z)$. We will use \cite[Lemma A.4.1]{behrens_EHP} to conclude that
\[d_{\beta}^{X}(\delta_{\beta}(w)) =- \delta_{\beta}(d_{\beta}^{Z}(w))\]
as follows.

\begin{figure}[t]
\center
\includegraphics[width=0.8\linewidth]{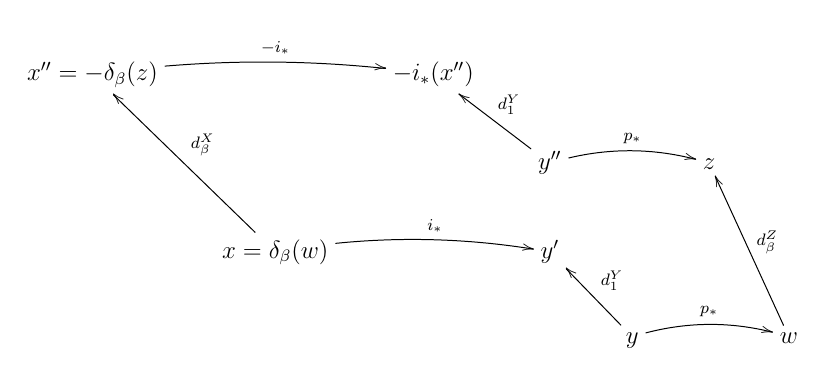}
\captionsetup{width=\textwidth}
\caption{Illustration of case (3) of \cite[Lemma A.4.1]{behrens_EHP}. 
}
\label{fig:gbt}
\end{figure}

Abusing notation, we choose representatives $w \in E_{t,\mu}^1(Z)$ and  $z \in E_{t-1, \mu-\beta}^{1}(Z)$. Since \eqref{eqn:se} is exact, there exists $y \in E_{t,\mu}^1(Y)$ such that $w=p_*(y)$. Let $y'= d_1^{Y}(y)$ and note that $y'\neq 0$. Otherwise, $y$ would be a permanent cycle since $E_{*,*}^*(Y)$ collapses at $E_2$, which would imply that $d_{\beta}^{Y}(y)=0$. This would contradict the fact that $d_{\beta}^Z(p_*(y)) =d_{\beta}^Z(w) \neq0$.

We will apply \cite[Lemma A.4.1]{behrens_EHP} to the differential $d_1^Y(y)=y'$. A step by step analysis of the proof shows that we are in Case~3 of \cite[Lemma A.4.1]{behrens_EHP}. We do not repeat the argument here as our treatment would be no simpler than the proof of \cite[Lemma A.4.1]{behrens_EHP}. In our case, we have $\alpha=\alpha''=1$, $\alpha'=\beta-1$, $x= \delta_\beta(w)$, $x'' = -\delta_{\beta}(z)$ where $\alpha,\alpha', \alpha''$, $x$ and $x''$ are as in the statement of Behrens' result. These elements are illustrated in \fullref{fig:gbt}.

Since $\delta_{\beta}$ anti-commutes with the differential $d_{\beta}$, it induces a map on the $E^{\beta+1}$--terms which we denote by $\delta_{\beta+1}$. 

Finally, we use Lemma A.4.1 again to prove that $\delta_{\infty}$ is a filtered version of $j_* \colon \pi_*Z \to \pi_{*-1}X$. 
Let $\overline{z} \in \pi_tZ$. Choose $z \in E^1_{t,\mu}(Z)$ which detects $\overline{z}$. Then there exists $y \in E^1_{t,\mu}(Y)$ such that $p_*(y) =z$. If $d_1^Y(y) =0$, then
\begin{itemize}
\item $y$ is a permanent cycle which represents a class $\overline{y} \in \pi_tY$ such that $p_*(\overline{y}) =\overline{z}$. Hence, $j_*(\overline{z}) =0$. 
\item $\delta_2(z)=0$ by the definition of the connecting homomorphism. Hence, $\delta_{\infty}(\overline{z})=0$.
\end{itemize}
In this case, $j_*(\overline{z}) = 0$ and $\delta_{\infty}(\overline{z}) = 0$, so the claim holds. So, suppose that $d_1^Y(y) = y' \neq 0$. We can apply Lemma A.4.1 to this differential, with $\alpha=1$. Since $p_*(\overline{y}) =\overline{z}$ is a permanent cycle, we are in Case~5, under the assumption that $p_*(y)$ detects $\overline{z}$. Then, there is a non-trivial element $x \in E^1_{t-1, \mu-1}$ such that $i_*(x) = y'$ and, by definition, $\delta_2(z) = x$. Further, $j_*(\overline{z})$ is detected by $x$, or it is detected in lower filtration if $x$ is the target of a differential in $E_{*,*}^*(X)$.
\end{proof}

\section{The monochromatic layer}\label{sec:mono}

In this section we give a description of the Tate spectral sequence of $\mono \E = M_nE_n$ in terms of the one for $\E$, which will be one of the key ingredients in the proof of \fullref{thm:main}.

Recall that 
\[\E_* = \W [\![ u_1, \ldots, u_{n-1} ]\!][u^{\pm1}] \]
for $\W = W(\F_{p^n})$ the Witt vectors on $\F_{p^n}$, $u_i \in \E_0$ for $1\leq i \leq n$ and $u \in \E_{-2}$. Further, we let 
\[v_k =
\begin{cases}  p & k=0 \\
u_ku^{1-p^k} & 1 \leq k <n \\
u^{1-p^n} & k=n. 
\end{cases}\]
Let $I_k$ be the ideal $(p,v_1,\ldots,v_{k-1})$ in $\E_*$, with the convention that $I_0$ is the zero ideal. Note that $v_k$ is invariant modulo $I_k$.
As in \cite[Section 5]{ravloc} or \cite{StrickGrossHop}, consider the cofiber sequences
\[\E/I_k^{\infty} \to v_k^{-1}\E/I_k^{\infty} \to \E/I_{k+1}^{\infty}.\]
One deduces from the fact that $\mono (v_k^{-1}\E/I_k^{\infty} ) \simeq \ast$ for $k<n$ that
\[\mono \E \simeq \Sigma^{-n}\mono (\E/I_{n}^{\infty}) \simeq \Sigma^{-n} \E/I_n^{\infty}. \] 
Recall that $\G_n$ acts on $\mono \E \simeq \E \smsh \mono S$ via its natural action on $\E$. Our goal in this section is to understand the cohomology of $C_p$ with coefficients $ \E_*/I_k^{\infty} $ in terms of its cohomology with coefficients in $\E_*/I_{k-1}^{\infty}$. To this end, we recall the following theorem, which is an adaptation of unpublished work of Hopkins and Miller.

\begin{notation}
Recall that $n=p-1$ and let $\zeta$ be a generator of $C_p$. For $0\leq k \leq n$, let $U_{k}$ be the $C_p$--module $\F_{p^n}\{z_n, z_{n-1}, \ldots, z_{k}\}$ 
with $\F_{p^n}$--linear action determined by $\zeta(z_i) = z_i + z_{i-1}$ for $k<i \leq n$ and $\zeta(z_{k}) =z_{k}$. Let $S_*(U_k)$ be the symmetric algebra on $U_k$ over $\F_{p^n}$ and let $d \in S_{p}(U_k)$ be given by $ d=\prod_{g\in C_p}g(z_n)$.
\end{notation}

\begin{thm}[Hopkins--Miller]\label{thm:getheory}
Let $\mathfrak{m}$ be the image of $I_n$ in $\E_{*}/I_k$. 
For $0<k\le n$, there is a map $\varphi\colon U_{k} \to \E_{-2}/I_k$ and $c_i \in \W^{\times}$ such that $\varphi(z_n) \equiv c_n u \mod (p, \mathfrak{m}^2)$ and $\varphi(z_i) \equiv c_iu u_i \mod (I_i, \mathfrak{m}^2)$ for $ k \leq i < n $.
The map $\varphi$ induces a $C_p$--equivariant map of algebras
\begin{equation*}
S_*(U_{k})[ d^{-1}] \to \E_{*}/I_k
\end{equation*} 
which becomes an isomorphism upon completion at $I = \varphi^{-1}(\mathfrak{m})$.  
\end{thm}

\begin{rem}
For a proof in the case of a maximal finite subgroup, see for example Nave's paper~\cite[Thm.~2.1]{nave1} and $k=0$. A similar choice of coordinates works for $C_p$, and the result for $k\ge 1$ follows easily from this. 
\end{rem}

The next result is the key input that fuels our group cohomology computations. 

\begin{lem}\label{lem:e0}
Let $H$ be $G$, $F$ (as defined in \eqref{eqn:G}) or $C_p$.
For $0 \leq k \leq n-1$ and $s >0$,
\[v_k^{k+1} \cdot H^s(H, \E_*/I_k) =0 \]
where $v_k = u_ku^{1-p^k} \in H^0(H, \E_*/I_k)$.
\end{lem}
\begin{proof}
The claims for $G$ and $F$ follow from the claim for $C_p$ by taking $G/C_p$ and $F/C_p$--fixed points since both quotients have order coprime to $p$. Further, if $k=0$, the claim follows from the fact that $pH^s(C_p, \E_*)=0$ if $s>0$; so we assume that $k\geq 1$. From \fullref{thm:getheory} we obtain
\[ 
H^*(C_p,  \E_{*}/I_k ) \cong (H^*(C_p,S_*(U_{k}))[d^{-1}])^{\wedge}_{I}.
\]
Note that the localization and completion can be moved outside, because $d$ and $I$ are invariant and our coefficients are finitely generated over $\E_*$. Moreover,
\[
\varphi(z_k) \equiv c_kv_ku^{p^k} \equiv c v_k \varphi(d)^{p^{k-1}} \mod (I_k, \mathfrak{m}^2)
\]
for some $c \in \W^{\times}$. Hence, $\varphi(z_k)$ and $v_k$ differ by a unit.
Therefore, it suffices to show that 
\[
z_k^{k+1}  H^s(C_p, S_*(U_k)) =0 
\]
for $s>0$.

For $t\geq 0$, $k+1\leq r \leq p-1$ and $s>0$, we show below that $H^s(C_p,S_{pt+r}(U_k))=0$.\footnote{In fact, it follows from Corollary 2.5 and Proposition 2.10 of \cite[Ch. III]{zpdecom} that for $t\geq 0$ and $k+1\leq r \leq p-1$,
$S_{pt+r}(U_k)$ is a free $\F_p[C_p]$--module.} 
Since multiplication by $z_k$ is a degree shifting map
\[
z_k\colon H^s(C_p, S_*(U_k)) \to  H^s(C_p, S_{*+1}(U_k)),   
\]
it follows that $z_k^{k+1}H^s(C_p, S_*(U_k)) =0$ for $s >0$, as desired.

The referee points out the following self-contained proof of the vanishing of $H^s(C_p,S_{pt+r}(U_k))$ in the cases stated above. Since $U_0$ is the regular representation of $C_p$ over $\F_{p^n}$, $H^0(C_p,S_*(U_0))$ is concentrated in degrees divisible by $p$ and $S_l(U_0)$ is free over $\F_{p^n}[C_p]$ whenever $p$ does not divide $l$. This implies that $H^s(C_p,S_l(U_0))=0$ for all $s>0$ and $l$ not divisible by $p$. For $k>0$, using the long exact sequences on cohomology associated to the short exact sequences 
\[
0 \to S_*(U_{k-1}) \xra{\cdot z_{k-1}} S_*(U_{k-1}) \to S_*(U_k) \to 0, 
\]
it follows that $H^s(C_p,S_{pt+r}(U_k)) = 0$ for all $s>0$, $t\ge 0$, and $k+1 \le r \le p-1$. 
\end{proof}

\begin{lem}\label{lem:vanishingMk}
Let $H$ be $G$, $F$ (as defined in \eqref{eqn:G}) or $C_p$.
For $0 \leq k \leq n-1$ and any $\E_*\E$--comodule $A$ which is $I_{k}$--power-torsion,
\[
\widehat{H}^*(H, v_{k}^{-1}A) = 0.
\]
In particular, this holds for $A = \E_*/I_k^{\infty}$.
\end{lem}
\begin{proof}
Again, we prove the claim for $C_p$ and note that the claim for $G$ and $F$ then follows by taking fixed points. Further, in the case $k=0$, the claim is that $H^s(C_p, p^{-1}A)=0$ for $s>0$, which holds because $|C_p|$ is invertible in $p^{-1}A$. So we assume that $k \geq 1$.

The group $C_p$ has periodic Tate cohomology, so it is enough to prove that ${H}^s(C_p, v_k^{-1}A )$ is zero for $s>0$. 
Let $\cE$ be the collection of those $I_k$--torsion $\E_*\E$--comodules $A$ for which $v_k^{-1}H^s(C_p,A) \cong H^s(C_p,v_k^{-1}A)$ vanishes for all $s>0$. On the one hand, $\cE$ is closed under internal shifts, extensions, and filtered colimits, since taking $C_p$--group cohomology commutes with filtered colimits. 
On the other hand, every $\E_*\E$--comodule can be written as a union over its finitely generated subcomodules. Moreover, if $A$ is $I_k$--power-torsion, so are its subcomodules. According to the Landweber filtration theorem, every finitely generated $I_k$--power-torsion $\E_*\E$--comodule is an iterated extension of cyclic comodules $\E_*/I_{k'}$ with $k' \ge k$, up to internal shifts. 

It follows that every $\E_*\E$--comodule $A$ which is $I_k$--power-torsion can be constructed from the set $\{\E_*/I_{k'}\}_{k' \ge k}$ via internal shifts, extensions, and filtered colimits. By \fullref{lem:e0}, $\E_*/I_k$ is in $\cE$, while $\E_*/I_{k'} \in \cE$ trivially for $k' > k$, so $A \in \cE$ as claimed.
\end{proof}

\begin{prop}\label{prop:hom}
Let $H$ be $G$, $F$ (as defined in \eqref{eqn:G}) or $H=C_p$.
There is an isomorphism of spectral sequences
\[
\xymatrix{ \widehat{H}^s(H, \pi_t(\mono \E)) \ar[r]^-{\cong} \ar@{=>}[d] &  \widehat{H}^{s+n}(H,\pi_{t+n}\E) \ar@{=>}[d]\\
\pi_{t-s} (\mono \E)^{tH} \ar[r]^-{\cong} & \pi_{t-s} \E^{tH}.}
\]
\end{prop}
\begin{proof}
We can filter the map $\mono \E \to \E$ as a composite
\[
\resizebox{\textwidth}{!}{
\xymatrix{\mono \E = \Sigma^{-n}\E/I_n^{\infty} \ar[r]^-{f_n} & \Sigma^{-(n-1)}\E/I_{n-1}^{\infty}  \ar[r]^-{f_{n-1}} & \cdots \ar[r]^-{f_{2}} & \Sigma^{-1}\E/I_{1}^{\infty} \ar[r]^-{f_{1}} & \E/I_{0}^{\infty} \simeq \E.}}
\]
On homotopy groups, the fiber sequences $\Sigma^{-k}\E/I_{k-1}^{\infty} \to \fib(f_k) \to \Sigma^{-k}\E/I_{k}^{\infty}$ induce short exact sequences
\begin{equation}\label{eq:chromaticses}
\xymatrix{0 \ar[r] & \E_*/I_{k-1}^{\infty} \ar[r] & v_{k-1}^{-1}\E_*/I_{k-1}^{\infty} \ar[r] & \E_*/I_{k}^{\infty} \ar[r] & 0.}
\end{equation}
By \fullref{lem:vanishingMk}, the $H$--group cohomology of the middle term vanishes in positive degrees, so we are in the situation of \fullref{lem:gbt}. It follows that the connecting homomorphism associated to \eqref{eq:chromaticses} gives an isomorphism of $H$--homotopy fixed point spectral sequences
\[
\xymatrix{H^s(H, \pi_t(\Sigma^{-k}\E/I_{k}^{\infty})) \ar[r]_-{\cong}^-{\delta_2} \ar@{=>}[d] & H^{s+1}(H,\pi_{t+1}(\Sigma^{-(k-1)}\E/I_{k-1}^{\infty})) \ar@{=>}[d]\\
\pi_{t-s} (\Sigma^{-k}\E/I_{k}^{\infty})^{hH} \ar[r]_-{\cong} & \pi_{t-s}(\Sigma^{-(k-1)}\E/I_{k-1}^{\infty})^{hH}
}
\]
for all $1 \le k \le n$. The claim about the Tate spectral sequences follows from this by the periodicity of Tate cohomology. 
\end{proof}

\section{Computation of the Gross--Hopkins dual shifts}\label{sec:final}

\subsection{The dual Tate spectral sequence}
Let $\I$ be the Gross--Hopkins dual of $S$ as defined in \fullref{sec:state}. Let $H$ be a finite subgroup of $\G$.
Our goal is to study $\I(\E^{hH})$. However, since the Tate construction $\E^{tH}$ vanishes, the norm $\E_{hH}\to \E^{hH}$ is a $\Kn$--local equivalence. It follows that 
\begin{equation*}
\I (\E^{hH}) \simeq \I (\E_{hH}) \simeq (\I\E)^{hH} \simeq (I_{\Q/\Z} \mono \E)^{hH}. \end{equation*}
We will therefore use the monochromatic layer $\mono \E$ to get information about the HFPSS of $\I(\E^{hH})$.

\begin{prop}\label{prop:dualcoh}
Let $H$ be $G$, $F$ (as defined in \eqref{eqn:G}) or $C_p$.
For $s\geq 1$, there is an isomorphism
\[ \widehat{H}^s(H, \pi_t \I\E)  \cong   D_{\Q/\Z} (  \widehat{H}^{-s-1+n}(H,\pi_{-t+n}\E))\]
compatible with differentials, in the sense that Tate differentials on the left hand side correspond to the Pontryagin duals of Tate differentials on the right hand side.
\end{prop}
\begin{proof}
\fullref{prop:dualhfss} provides an isomorphism between the HFPSS for $(\I\E^{hH}) \simeq (I_{\Q/\Z} \mono \E)^{hH}$ and the Pontryagin dual of the HOSS computing $(\mono \E)_{hH}$. In particular, there is an isomorphism of $E_2$--terms
\[
H^s(H, \pi_t\I\E) \cong H^s(H, \pi_tI_{\Q/\Z}\mono \E) \cong D_{\Q/\Z}(H_s(H,\pi_{-t}\mono \E))
\]
for all $s>0$, compatible with differentials. (See \fullref{sec:prelim} for this compatibility under the second isomorphism.) Using the conventions of \fullref{rem:tate} and \fullref{prop:hom}, upon Tate periodization we obtain the following isomorphism of the TSS for $\I\E$ and the Pontryagin dual of the TSS for $\E$
\[
\widehat{H}^s(H, \pi_t\I\E) \cong D_{\Q/\Z}(\widehat{H}^{-s-1}(H,\pi_{-t}\mono \E)) \cong D_{\Q/\Z}(\widehat{H}^{-s-1+n}(H,\pi_{-t+n}\E)).
\]
\end{proof}

\begin{rem}\label{rem:remark42}
Note that in practice, \fullref{prop:dualcoh} gives a procedure for obtaining the TSS for $\I\E$ from that for $\E$ and vice versa:  One needs to shift and rotate, i.e., perform the transformation of the plane $(s,t) \mapsto (-s-1+n, -t+n) $ to go from one to the other. In the next subsection, we recall the relevant computations for the TSS for $\E$, which will determine fully the TSS for $\I\E$. However, that alone does not suffice for determining the HFPSS for $\I\E$, as we cannot directly relate the zero-line to the one of the HFPSS for $\E$. To do that, we use Gross-Hopkins duality in \fullref{sec:DualCompute}.
\end{rem}

\subsection{Recollections on Tate cohomology}
In order to proceed from here, recall the following result, due to Hopkins and Miller; published accounts of it can be found in \cite{nave1,symtate}. An illustration in the case $p=5$ is given in \fullref{fig:tate5} and \fullref{fig:tate5dual}. We use the notation $x \doteq y$ if $x=\lambda y$ for a unit $\lambda$ in $\F_{p^n}$.
\begin{thm}[Hopkins--Miller]
\label{prop:coheop}
There is an isomorphism 
\[\widehat{H}^*(C_p, \E_*) \cong \F_{p^n}[a,b^{\pm 1}, \delta^{\pm 1}]/(a^2) \]
where the degrees $(s,t)$, for $s$ the cohomological and $t$ the internal degree, are given by $|a|=(1,-2)$, $|b| = (2,0)$ and $|\delta|=(0,2p)$.
The differentials in the spectral sequence are determined by
\begin{align*} d_{2n+1}(\delta) &\doteq ab^n \delta^{2}  & d_{2n^2+1}(a) &\doteq b^{n^2+1}\delta^{n-1}
\end{align*}
and the fact that they are $a\delta$, $b\delta^{n}$ and $\delta^p$--linear. The homotopy groups of $\E^{hC_p}$ are periodic of period $2p^2$ on the element $\delta^p$.
\end{thm}

\begin{rem}
From \fullref{prop:coheop}, we see that the $E_2$--term of the spectral sequence 
\[\widehat{H}^s(C_p, \E_t)  \Longrightarrow  \pi_{t-s}\E^{tC_p}\]
is an $\F_{p^n}$--vector space on classes $\delta^{k-r}b^r$ and $a \delta^{k-r}b^r$. 
The $d_{2n+1}$--differentials are 
\begin{align*}
 d_{2n+1}(\delta^{k-r}b^r) \doteq 
 k a\delta^{k-r+1} b^{n+r} 
 \end{align*}
for $k \not\equiv 0 \mod (p)$.
The $E_{2p}$--term of the Tate spectral sequence
is thus generated by terms of the form $a \delta^{k-r}b^r$ and $\delta^{k-r}b^r$ where $k \equiv 0 \mod (p)$.
The remaining non-trivial differentials are
\begin{align*}
d_{2n^2+1}(a \delta^{k-r}b^r) & \doteq
 \delta^{k-r+n-1}b^{n^2+1+r}.
\end{align*}
\end{rem}

A proof of the following results can be found in \cite{heard_eop}.

\begin{thm}[Hopkins--Miller]\label{thm:cohF}
There is an isomorphism
\[\widehat{H}^*(F, \E_*) \cong \F_{p^n}[\alpha, \beta^{\pm 1}, \Delta^{\pm 1}]/(\alpha^2),\]
where $\alpha = \delta a$, $\beta = b \delta^n$ and $\Delta = \delta^p$, so that
$|\alpha| = (1,2n)$, $\beta = (2,2pn)$ and $|\Delta| = (0, 2pn^2)$.
In the TSS, there are differentials
\begin{align*}d_{2n+1}(\Delta) &\doteq \alpha \beta^{n},  & d_{2n^2+1}(\alpha\Delta^n) &\doteq \beta^{n^2+1} \end{align*}
which determine all other differentials. The homotopy groups of $\E^{hF}$ are periodic of period $2n^2p^2$ on the element $\Delta^p$.
\end{thm}

\begin{thm}[Hopkins--Miller]\label{thm:HMTateSS}
There is an isomorphism
\[\widehat{H}^*(G, \E_*) \cong \F_{p}[\alpha, \beta^{\pm 1}, \Delta^{\pm 1}]/(\alpha^2)\]
with differentials as in \fullref{thm:cohF}.
\end{thm}

\subsection{Computations of the duals}\label{sec:DualCompute}

We now turn to the proof of our main results, which are illustrated in \fullref{fig:tate5}, \fullref{fig:tate5dual} and \fullref{fig:tate5G}. The next few results are used to show that the spectral sequences for $\I\E^{hH}$ are isomorphic to shifts of that for $\E^{hH}$. Similar results are also discussed in \cite[Section 5]{heard_eop}, but we give more details here.

\begin{lem}\label{lem:moduless}
Suppose $\ssM_*^{**}$ is a spectral sequence of modules over a spectral sequence $\ssE_*^{**}$ of algebras. Assume additionally that for some $r \ge 2$, the $r$th page $\ssM_r^{**}$ is a cyclic module over $\ssE_r^{**}$ , i.e., that there exists a class $x \in \ssM_r^{**}$ such that $\ssM_r^{**} \cong \ssE_r^{**}\langle x \rangle$. If $x$ is a $d_r$--cycle, then $\ssM_{r+1}^{**} \cong \ssE_{r+1}^{**}\langle x \rangle$ is a cyclic module on the class $x \in \ssM_{r+1}^{**}$ corresponding to $x$.

In particular, if also $x$ is not a boundary and $\ssM_r^{**}$ is free of rank one on $x$, then $\ssM_{r+1}^{**}$ is also free of rank one on $x$.
\end{lem}
\begin{proof}
First, the differentials $d_r^\ssM$ on the $r$th page of $\ssM_*^{**}$ are determined by the differentials $d_r^\ssE$ for $\ssE_r^{**}$: Indeed, by assumption any element in $\ssM_r^{**}$ can be written as $a\cdot x$ for $a \in \ssE_r^{**}$, so that
\[
d_r^\ssM(a\cdot x) = d_r^\ssE(a)\cdot x \pm a \cdot d_r^\ssM(x) = d_r^\ssE(a)\cdot x.
\]
If $x$ is not a boundary, this implies the claim, so it remains to consider the case that $x$ is a boundary on $\ssM_r^{**}$. Suppose that there exists $y$ such that $d_r^\ssM(y) = x$. Let $a\cdot x$ be an arbitrary cycle on $\ssM_r^{**}$, i.e., $d_r^\ssE(a) = 0$. We thus get 
\[
d_r^\ssM(a\cdot y) = d_r^\ssE(a) \cdot y \pm a \cdot d_r^\ssM(y) = a \cdot x,
\]
so that $a\cdot x$ is a boundary as well, hence $\ssM_{r+1}^{**} = 0$. 
\end{proof}

To avoid confusion about which spectral sequence we are referring to, we will say ``the TSS for $X\circlearrowleft G$'' to refer to the TSS whose $E_2$-page is $\widehat{H}^*(G, \pi_*X)$, converging to $\pi_*X^{tG}$, and similarly in the case of the HFPSS and HOSS.

\begin{prop}\label{prop:diffsforced}
Let $X$ be a spectrum with a $G$--action. Suppose that $\pi_*X$ is free of rank one as a $G$--$\E_*$--module, and that the TSS for $X \circlearrowleft H$ is a module over the TSS for $\E \circlearrowleft H$ whenever $H$ is a subgroup of $G$. Then the TSS for $X \circlearrowleft H$ is, up to a shift, isomorphic to the TSS for $\E\circlearrowleft H$.
\end{prop}

\begin{rem}
Of course, the prime example of such a situation is when $X$ is itself an $\E$--module with a compatible $G$--action.
\end{rem}

\begin{proof}
First, note that if $H$ does not contain the $p$--torsion, the spectral sequences collapse and the claim follows trivially. So we assume that $C_p$ is a subgroup of $H$. Note that, by \fullref{lem:moduless}, it suffices to show that there is a module generator $x$ which is a permanent cycle.

First, we prove that such an $x$ exists when $H=G$, using the explicit knowledge of \fullref{thm:HMTateSS}.
The assumption implies that there exists $r\in \Z$ and a module generator $y$ in degree $(0,0)$, so that
\[\widehat{H}^*(G, \pi_*\Sigma^rX)  \cong \widehat{H}^*(G, \pi_*\E)\{y\}.\] 
For degree reasons, $y$ is a $d_{2n}$--cycle. Suppose that $d_{2n+1}(y) \ne 0$. This implies that $d_{2n+1}(y) = \lambda  \alpha \beta^n\Delta^{-1} y$ for some $\lambda \in \F_{p}^{\times}$. We know that $d_{2n+1}(\Delta) =  \gamma \alpha \beta^n$ for $\gamma \in \F_p^{\times}$, so we compute
\[
d_{2n+1}(\Delta^ky) =  (k\gamma+\lambda)  \alpha \beta^n \Delta^{k-1} y.
\]
Take $k$ to be an integer such that $k \equiv -\lambda/\gamma$ modulo $(p)$, so that $x = \Delta^{k}y$ is a $d_{2n+1}$--cycle. Since $\Delta$ is invertible, $x$ is a free generator of the $E_2$--term of the TSS for $X$  over $\widehat{H}^*(G, \pi_*\E)$, so we can apply \fullref{lem:moduless}. Proceeding inductively, one sees that there can be no higher differentials on $x$.

Note that, since $C_p\subseteq H$, $|G/H|$ is coprime to $p$. Hence, 
\[ \widehat{H}^*(G, \pi_*\Sigma^r X) \cong \widehat{H}^*(H, \pi_*\Sigma^rX)^{G/H} \]
and the image of $x$ in $\widehat{H}^*(H, \pi_*\Sigma^rX)$ is a permanent cycle which generates the $E_2$--term as an $\widehat{H}^*(H, \E_*)$--module. 
\end{proof}

\begin{rem}\label{rem:shift}
For $X$ as in \fullref{prop:diffsforced}, it then follows from \fullref{rem:tate} that the HFPSS and HOSS for $X \circlearrowleft H$ are isomorphic to the corresponding ones for $\E \circlearrowleft H$ up to the same shift.
Further, to determine the shift, it suffices to find a $d_{2n+1}$--cycle $x$ on the zero line of the TSS for $X \circlearrowleft H$. If this element $x$ has degree $(s,t)=(0,r)$, the spectral sequence for $X^{hH}$ will be isomorphic to the spectral sequence for $\Sigma^r \E^{hH}$. To contrast with \fullref{rem:remark42}, in this case, the assumptions on $X$ give an identification of the $E_2$-pages of the HFPSS and HOSS. This identification of $E_2$-pages together with the identification of differentials in the TSS allows us to relate them to the HFPSS and HOSS for $\E \circlearrowleft H$.
\end{rem}

\begin{cor}\label{cor:ssshifts}
Let $H$ be a subgroup of $G$. There are integers $k_D^H$ and $k_I^H$ so that the TSSs, HFPSSs and HOSSs for $\Dn\E \circlearrowleft H$ and $\I\E \circlearrowleft H$ are isomorphic to the corresponding spectral sequences for $\Sigma^{k_D^H}\E \circlearrowleft H$ and $\Sigma^{k_I^H}\E \circlearrowleft H$, respectively. Further, $\Dn\E^{hH}\simeq \Sigma^{k_D^H}\E^{hH}$ and $\I\E^{hH} \simeq \Sigma^{k_I^H}\E^{hH}$, hence $\Dn\E^{hH}, \I\E^{hH} \in \Pic(\E^{hH})$. 
\end{cor}
\begin{proof}
Let $X = \Dn\E$ or $\I\E$. Strickland~\cite{StrickGrossHop} shows that $\pi_*\I\E \cong \Sigma^{-n}\E_*{\left<\det\right>}$ and $\pi_*\Dn\E \cong \Sigma^{n^2}\E_*$ as $\mathbb{G}$-$\E_*$--modules. 
(See \eqref{eqn:actdet} for a description of the Morava module $\E_*{\left<\det\right>}$.)
Further, we show below in \fullref{lem:isodeltak} that $ \Sigma^{-n}\E_*{\left<\det\right>}$ is isomorphic to a shift of $\E_*$ as a $G$--module. Therefore, $X$ satisfies the conditions of \fullref{prop:diffsforced}, and the claim for the spectral sequences follows. Let $r= k_D^H$ for $\Dn\E^{hH}$ or $r=k_I^H$ for $\I\E^{hH}$ be the appropriate shift (as in \fullref{rem:shift}).

The construction of the HFPSS for $X$ is compatible with the $\E^{hH}$--module structures.  
This implies that the abutment of the spectral sequence for $X^{hH}$ is isomorphic to $\pi_*\Sigma^{r}\E^{hH}$ as a $\pi_*\E^{hH}$--module. From this, we conclude that the unit of $\E^{hH}$ induces an equivalence 
$\Sigma^{r}\E^{hH} \xra{\simeq} X^{hH}$,
as claimed. 
\end{proof}

We now prove the first part of \fullref{main:intro}, which is illustrated in \fullref{fig:tate5} and \fullref{fig:tate5dual}.
 
\begin{thm}\label{thm:main}
For $p\geq 3$, $n=p-1$, and $C_p$ a finite $p$--torsion subgroup of $\G$, there is an equivalence $\I\E^{hC_p} \simeq \Sigma^{n^2} \E^{hC_p}$.
\end{thm}

\begin{proof}
By \fullref{cor:ssshifts}, the spectral sequence for $\I\E^{hC_p}$ is isomorphic to a shift of the spectral sequence for $\E^{hC_p}$. 
Note that, since $\Z_p^{\times}$ contains no $p$--torsion, $C_p$ is in the kernel of $\det\colon \G_n \to \Z_p^{\times}$, so that $\pi_* \I\E \cong  \pi_* \Sigma^{-n} \E$ as $C_p$--$\E_*$--modules, thus  
\begin{equation*}
  {H}^{*}(C_p, \pi_{*} \Sigma^n \I\E)   \cong  {H}^{*}(C_p, \pi_{*} \E).
\end{equation*}
We will show that the differentials for the spectral sequence of $\Sigma^n \I\E^{hC_p}$ are the same as those for the spectral sequence for $\Sigma^{np}\E^{hC_p}$. The claim will then follow from \fullref{cor:ssshifts}, noting that $np-n = n^2$.

By \fullref{rem:shift}, it suffices to locate a $d_{2n+1}$--cycle on the zero line of the HFPSS to determine the shift. 
Via the correspondence of \fullref{prop:dualcoh}, we have an isomorphism
\begin{align*}
D_{\Q/\Z}(  \widehat{H}^{s}(C_p,\E_{t}))  \cong \widehat{H}^{n-(s+1)}(C_p, \pi_{2n-t} \Sigma^n \I\E),
\end{align*}
which extends to an isomorphism of spectral sequences. The class $\varepsilon = ab^{\frac{n}{2}-1} \delta^{-\frac{n}{2}+1} $ in $H^{n-1}(C_p, \E_{2n-np})$ is a $d_{2n+1}$--cycle which supports a $d_{2n^2+1}$--differential. Therefore, $\varepsilon^*=D_{\Q/\Z}(\varepsilon)$ is a class in 
$\widehat{H}^{0}(C_p, \pi_{np} \Sigma^n \I\E) $, which is a $d_{2n+1}$--cycle. This proves the claim.
\end{proof}

\captionsetup{width=\linewidth}
\begin{figure}[t]
\centering
\includegraphics[width=\linewidth]{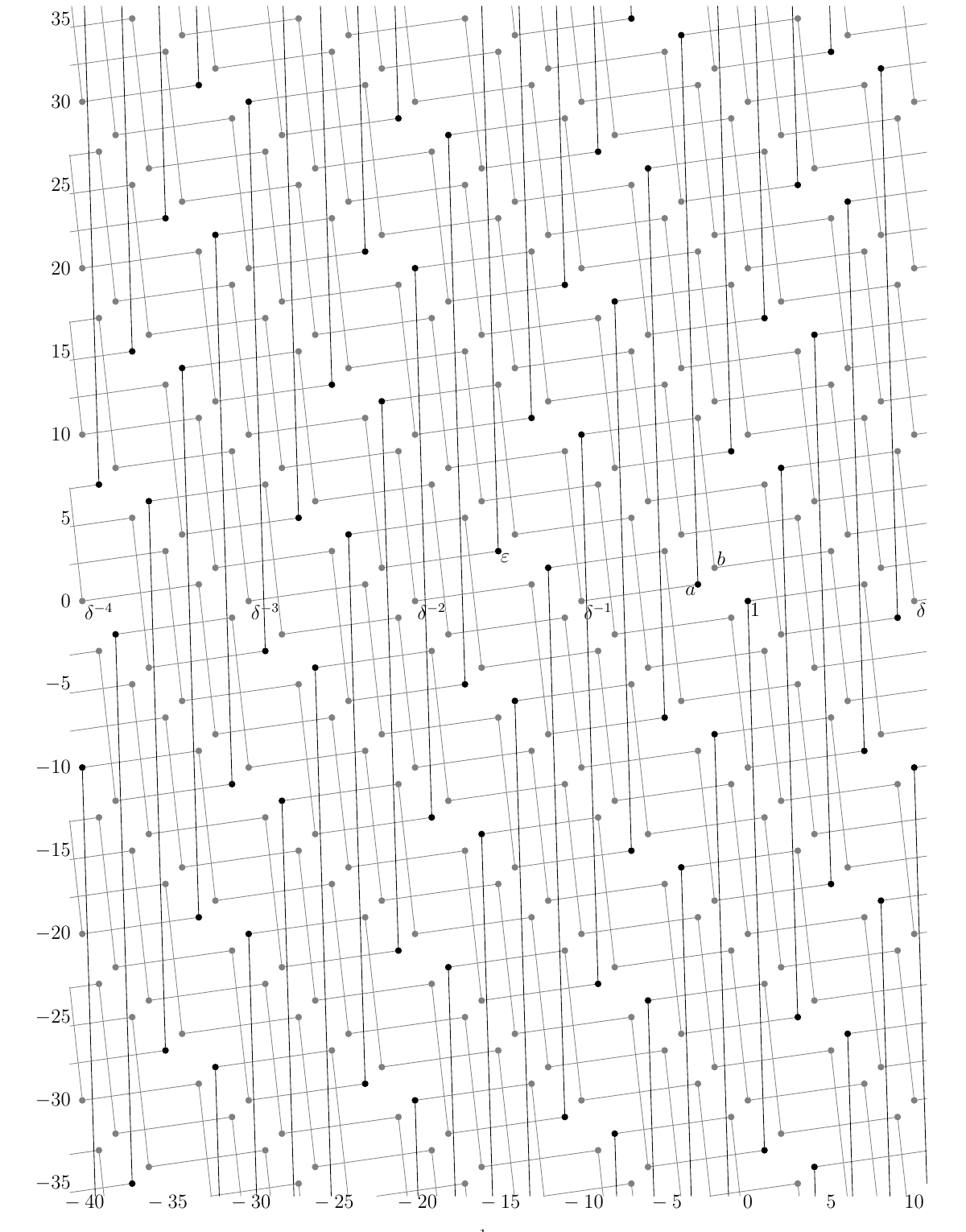}
\caption{The Tate spectral sequence $\widehat{H}^s(C_p, \pi_t \E) \Rightarrow \pi_{t-s}\E^{tC_p}$.
for $p=5$. A $\bullet$ denotes a copy of $\F_{p^n}$. The $d_{2n+1}$ differentials are in gray and the $d_{2n^2+1}$ are in black. See also \fullref{fig:tate5closeup}.}
\label{fig:tate5}
\end{figure}

\captionsetup{width=\linewidth}
\begin{figure}[t]
\centering
\includegraphics[width=\linewidth]{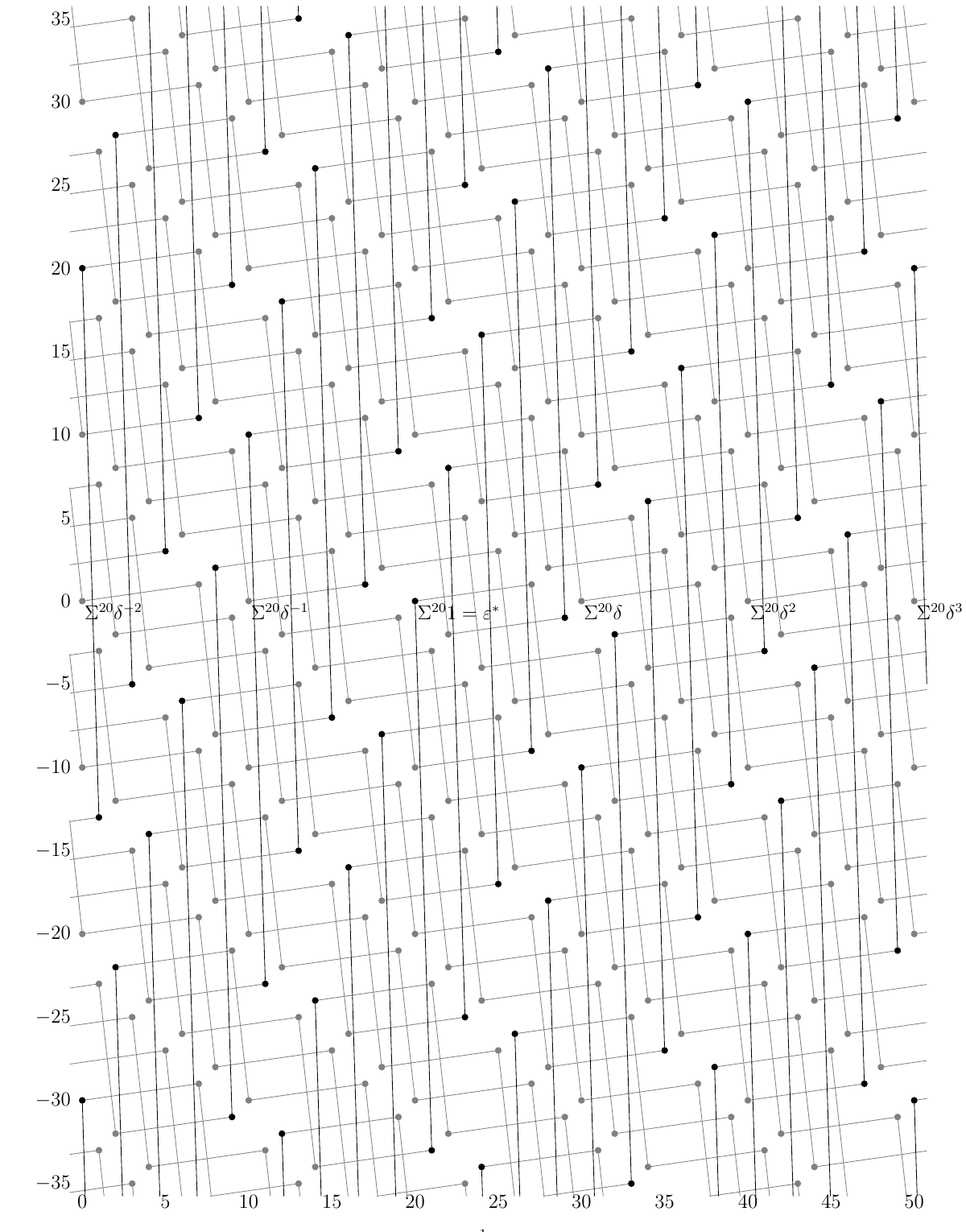}
\caption{The Tate spectral sequence $\widehat{H}^s(C_p, \pi_t\Sigma^n \I\E) \Rightarrow \pi_{t-s}(\Sigma^n \I\E)^{tC_p}$
for $p=5$. A $\bullet$ denotes a copy of $\F_{p^n}$. The $d_{2n+1}$ differentials are in gray and the $d_{2n^2+1}$ are in black. See also \fullref{fig:tate5closeup}.}
\label{fig:tate5dual}
\end{figure}

%%%%
\begin{figure}
\center
\includegraphics[width=\linewidth]{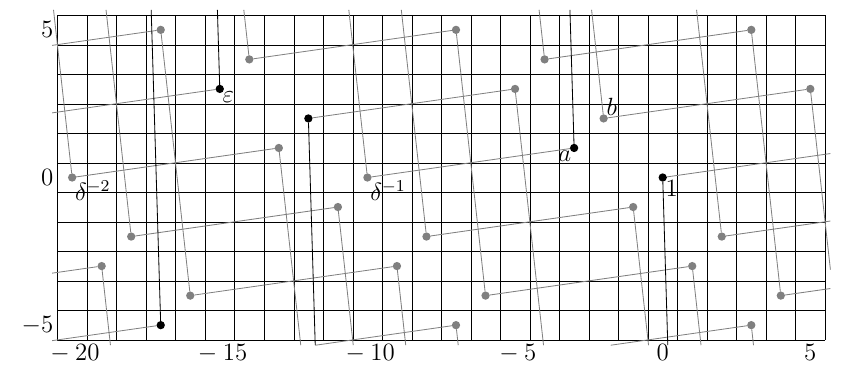}
\includegraphics[width=\linewidth]{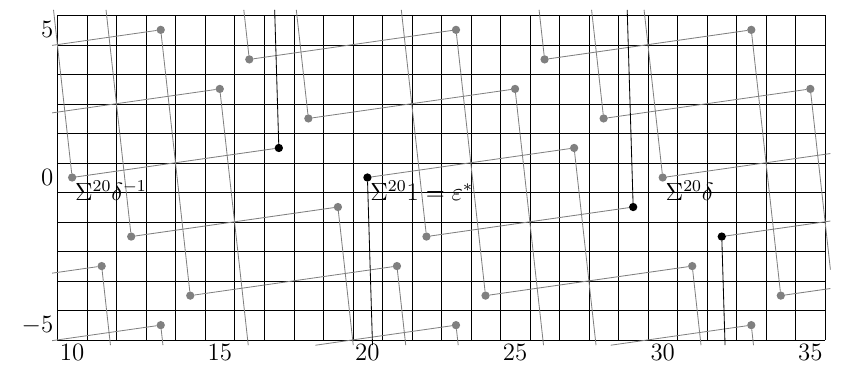}
\captionsetup{width=0.8\linewidth}
\caption{Closeup views of the Tate spectral sequences of \fullref{fig:tate5} (top) and \fullref{fig:tate5dual} (bottom).}
\label{fig:tate5closeup}
\end{figure}
%%%%%

Next, we turn to the groups $F$ and $G$ defined in \eqref{eqn:F} and \eqref{eqn:G}, respectively. The group $F \cong C_p \rtimes C_{n^2}$ is a subgroup of $\mathbb{S}_n$ which is not in the kernel of the determinant.
Therefore, since $\pi_*\Sigma^n \I\E \cong \Edet$ we must study
$\E_*{\left<\det\right>} = \E_* \widehat{\otimes}_{\Z_p} \Z_p{\left<\det\right>}$
as an $F$--module. We will prove the following result.
\begin{prop}\label{prop:IEF}
Let $k = -\frac{n}{2}(n-2)$. As $F$--$\E_*$--modules, $\pi_*\Sigma^n\I\E \cong \pi_*\Sigma^{2pk}\E$. 
\end{prop}

In general, for $g\in \G$ and $x\in \E_*$, we write
\begin{align}\label{eqn:actdet}
\widehat{g}(x) = g(x) \det(g)
\end{align}
for the action of $g$ on $\Edet$. Let $\omega \in \W$ be a primitive $p^n-1$--root of unity. Viewing $\W^{\times}$ as a subgroup of $\mathbb{S}_n$, $F$ is generated by $\zeta \in C_p$ and
$\tau = \omega^{\frac{p^{n}-1}{n^2}}$. We write $\eta =  \omega^{\frac{p^n-1}{n^2}}$ when viewing $\W \subseteq \E_0$. Then $e= \eta^{n}$ is a generator of the $n$--torsion in $\Z_p^{\times} \subseteq \E_0$. Further, 
\[\det(\omega) = \prod_{r=0}^{{n-1}}\omega^{p^r} = e,\] 
so it follows that
\[\det(\tau) = e^{\frac{p^n-1}{n^2}} = \eta^{\frac{p^n-1}{n}}.\]
Using the fact that $\det(\zeta)=1$, we collect these observations in the following lemma.
\begin{lem}
For $x \in \E_*{\left<\det\right>}$, the action of $F$ on $x$ is determined by
\begin{align*}
\widehat{\zeta}(x) &= \zeta(x) \\
 \widehat{\tau}(x) &= \tau(x) \eta^{\frac{p^n-1}{n}} . 
\end{align*}
\end{lem}
\begin{rem}
If $p=3$, then $\frac{p^n-1}{n^2}=2$ and $e^2=1$. So, in this case (but not more generally) $F$ in fact is in the kernel of the determinant. 
\end{rem}
It is shown in 
\cite[Lemma 5.3]{heard_eop} that $\tau(\delta) = \eta^{-p}\delta$.
We use this to prove the following result.

\begin{lem}\label{lem:basicidentity}
The class $\delta^k \in \E_*{\left<\det\right>}$, where $k= -\frac{n}{2}(n-2)$, is invariant under the action of $F$.
\end{lem}
\begin{proof}
Since $\delta$ is invariant under the action of $C_p$, it is sufficient to check that $\widehat{\tau}(\delta^k) = \delta^k$. We have
\[\widehat{\tau}(\delta^{k}) = \eta^{ \frac{n}{2}(n-2)p +\frac{p^n-1}{n} }\delta^{k}. \]
Since $\eta^{n^2} \equiv 1$, we must show that 
$\frac{n}{2}(n-2)p +\frac{p^n-1}{n}\equiv 0$ modulo $(n^2)$.
However, modulo $(n^2)$,
\begin{align*}
\frac{p^n-1}{n} = \frac{(n+1)^n-1}{(n+1)-1} = \sum_{i=0}^{n-1}(n+1)^i \equiv  \sum_{i=0}^{n-1}(in+1) = \frac{n\cdot n(n-1)}{2}+n. 
\end{align*}
Therefore,
\begin{align*}
 \frac{n}{2} (n-2)p +\frac{p^n-1}{n}
&\equiv  \frac{n}{2} (n-2)p + \frac{n\cdot n(n-1)}{2}+n \\
&= \frac{n\cdot n(n-1)}{2}+\frac{n\cdot n(n-1)}{2} \equiv 0 \mod (n^2). \qedhere
\end{align*}
\end{proof}

\begin{lem}\label{lem:isodeltak}
Let $k = -\frac{n}{2}(n-2)$. 
The map $\varphi\colon \Sigma^{2pk}\E_{*} \to \E_{*}\langle {\det}\rangle $ defined by $\varphi(x) = x \delta^k$ is an isomorphism of $\E_*$--$F$--modules.
\end{lem}
\begin{proof}
Since $\delta$ is a unit in $\E_{*}$, this is an isomorphism of $\E_*$--modules. Further, $\delta$ is fixed by the action of $C_p$ and, since $C_p$ is in the kernel of the determinant, $\varphi$ is an isomorphism of $C_p$--modules. So, it suffices to check that $\varphi$ is equivariant under the action of $\tau$. 

However, 
\begin{align*}
\widehat{\tau}(\varphi(x))  &=  \widehat{\tau}(x\delta^k) \\
& = \tau(x \delta^k) \det(\tau) \\
&=\tau(x) \tau( \delta^k) \det(\tau)  \\
&= \tau(x) \widehat{\tau}(\delta^k).
\end{align*}
By \fullref{lem:basicidentity}, $ \widehat{\tau}(\delta^k)  =\delta^k$. So, we can conclude that 
\[\widehat{\tau}(\varphi(x)) = \tau(x)\delta^k = \varphi(\tau(x)),\]
as desired.
\end{proof}

\fullref{prop:IEF} is then a direct consequence of \fullref{lem:isodeltak}. We can now prove the second part of \fullref{main:intro}, which is illustrated in \fullref{fig:tate5G}.

\begin{thm}\label{thm:mainF}
For $p\geq 3$, $n=p-1$, and $H$ the subgroup $G$ or $F$ of $\G$ as defined in \eqref{eqn:G}, there is an equivalence 
$I(\E^{hH}) \simeq \Sigma^{np^2+n^2 } \E^{hH}$.
\end{thm}

\begin{proof}
The claim for $G$ follows from that for $F$ by taking Galois invariants (see \cite[Section 1.4]{goerss_bobkova} for details on the passage to Galois fixed points). 
Again, let $k = -\frac{n}{2}(n-2)$. By \fullref{prop:IEF}, 
\[H^s(F, \pi_t\Sigma^n \I\E) 
\cong H^s(F, \pi_t\Sigma^{2pk}\E). \]
Further, from the $2pn^2$--periodicity of $H^s(F, \pi_t\E)$ with respect to $t$, it follows that the $E_2$--term for $\Sigma^n \I\E^{hF}$ is isomorphic to that for $\Sigma^{2pn^2+2pk}\E^{hF} = \Sigma^{2np+n^2p}\E^{hF} $. We prove that these spectral sequences have the same differentials. The result then follows by \fullref{cor:ssshifts}. 

Again, it suffices to locate a $d_{2n+1}$--cycle on the zero line of the TSS for $\Sigma^n \I\E^{hF}$.
From \fullref{prop:dualcoh}, there is an isomorphism
\[D_{\Q/\Z}(  \widehat{H}^{s}(F,\E_{t}))  \cong \widehat{H}^{n-(s+1)}(F, \pi_{2n-t} \Sigma^n \I\E) ,\]
which extends to an isomorphism of spectral sequences. The class $\varepsilon = \alpha \beta^{\frac{n}{2}-1}\Delta^{-1} $ in $H^{n-1}(F, \E_{2n-2np-n^2p})$ is a $d_{2n+1}$--cycle which supports a $d_{2n^2+1}$--differential. Therefore, $\varepsilon^*=D_{\Q/\Z}(\varepsilon)$ is a class in 
$\widehat{H}^{0}(F, \pi_{2np+n^2p} \Sigma^n \I\E) $, which is a $d_{2n+1}$--cycle. Noting that $2np+n^2p-n = np^2+n^2 $ proves the claim.
\end{proof}

\newpage
%%%%%
\begin{figure}[H]
\center
\includegraphics[height=0.45\textheight]{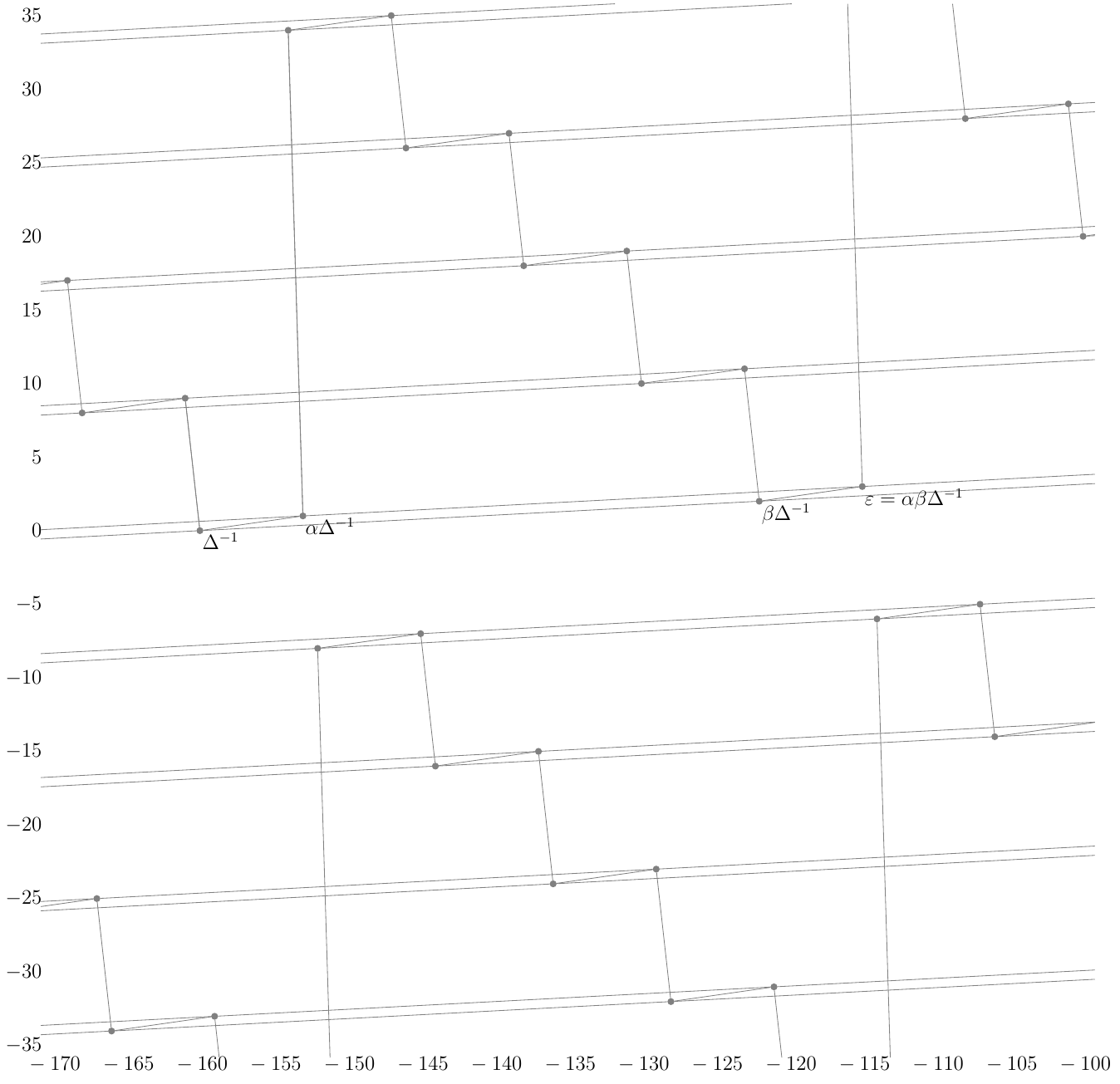}\\
\vspace{1cm}
\includegraphics[height=0.45\textheight]{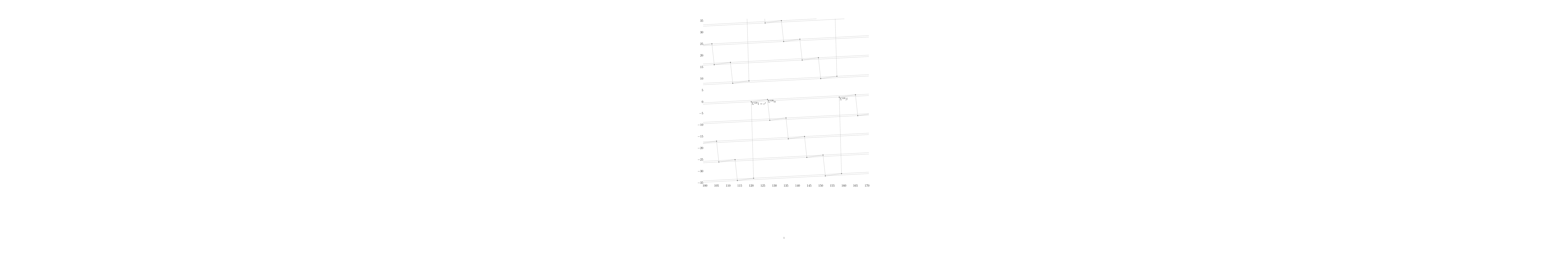}
\captionsetup{width=0.8\textwidth}
\caption{The Tate spectral sequences $\widehat{H}^s(F, \pi_t \E) \Rightarrow \pi_{t-s}\E^{tF}$ (top) and $\widehat{H}^s(F, \pi_t\Sigma^n \I\E) \Rightarrow \pi_{t-s}(\Sigma^n \I\E)^{tF}$ (bottom), for $p=5$. A $\bullet$ denotes a copy of $\F_{p^n}$. }
\label{fig:tate5G}
\end{figure}
%%%%%

%
%%%%%%%%%%%%%%%%%%%%%

%%%%%%%%%%%%%%

%\bibliographystyle{amsalpha}
%\bibliography{bib}

\begin{thebibliography}{GHMR05}

\bibitem[AF78]{zpdecom}
Gert Almkvist and Robert Fossum, \emph{Decomposition of exterior and symmetric
  powers of indecomposable {${\bf Z}/p{\bf Z}$}-modules in characteristic
  {$p$}\ and relations to invariants}, S\'eminaire d'{A}lg\`ebre {P}aul
  {D}ubreil, 30\`eme ann\'ee ({P}aris, 1976--1977), Lecture Notes in Math.,
  vol. 641, Springer, Berlin, 1978, pp.~1--111. \MR{499459}

\bibitem[And69]{Anderson}
DW~Anderson, \emph{{Universal coefficient theorems for $K$-theory, mimeographed
  notes}}, Univ. California, Berkeley, Calif (1969).

\bibitem[BBGS18]{det}
T.~{Barthel}, A.~{Beaudry}, P.~G. {Goerss}, and V.~{Stojanoska},
  \emph{{Constructing the determinant sphere using a Tate twist}}, ArXiv
  e-prints (2018).

\bibitem[Beh06]{behrens_mod}
Mark Behrens, \emph{A modular description of the {$K(2)$}-local sphere at the
  prime 3}, Topology \textbf{45} (2006), no.~2, 343--402. \MR{2193339}

\bibitem[Beh12]{behrens_EHP}
\bysame, \emph{The {G}oodwillie tower and the {EHP} sequence}, Mem. Amer. Math.
  Soc. \textbf{218} (2012), no.~1026, xii+90. \MR{2976788}

\bibitem[BG]{goerss_bobkova}
Irina Bobkova and Paul~G. Goerss, \emph{Topological resolutions in
  {$K(2)$}-local homotopy theory at the prime {$2$}}, Journal of Topology
  \textbf{11}, no.~4, 917--956.

\bibitem[GH16]{GH_bcdual}
Paul~G. Goerss and Hans-Werner Henn, \emph{The {B}rown-{C}omenetz dual of the
  {$K(2)$}-local sphere at the prime 3}, Adv. Math. \textbf{288} (2016),
  648--678. \MR{3436395}

\bibitem[GHMR05]{ghmr}
P.~Goerss, H.-W. Henn, M.~Mahowald, and C.~Rezk, \emph{A resolution of the
  {$K(2)$}-local sphere at the prime 3}, Ann. of Math. (2) \textbf{162} (2005),
  no.~2, 777--822.

\bibitem[GM95]{GreenleesMay}
J.~P.~C. Greenlees and J.~P. May, \emph{Generalized {T}ate cohomology}, Mem.
  Amer. Math. Soc. \textbf{113} (1995), no.~543, viii+178. \MR{1230773}

\bibitem[GS96]{greensad_tate}
J.~P.~C. Greenlees and Hal Sadofsky, \emph{The {T}ate spectrum of
  {$v_n$}-periodic complex oriented theories}, Math. Z. \textbf{222} (1996),
  no.~3, 391--405. \MR{1400199}

\bibitem[{Hea}15]{heard_eop}
D.~{Heard}, \emph{{The Tate spectrum of the higher real $K$-theories at height
  $n=p-1$}}, ArXiv e-prints (2015).

\bibitem[Hen07]{henn_res}
Hans-Werner Henn, \emph{On finite resolutions of {$K(n)$}-local spheres},
  Elliptic cohomology, London Math. Soc. Lecture Note Ser., vol. 342, Cambridge
  Univ. Press, Cambridge, 2007, pp.~122--169. \MR{2330511}

\bibitem[HG94a]{grosshopkins2}
M.~J. Hopkins and B.~H. Gross, \emph{Equivariant vector bundles on the
  {L}ubin-{T}ate moduli space}, Topology and representation theory ({E}vanston,
  {IL}, 1992), Contemp. Math., vol. 158, Amer. Math. Soc., Providence, RI,
  1994, pp.~23--88. \MR{1263712}

\bibitem[HG94b]{grosshopkins1}
\bysame, \emph{The rigid analytic period mapping, {L}ubin-{T}ate space, and
  stable homotopy theory}, Bull. Amer. Math. Soc. (N.S.) \textbf{30} (1994),
  no.~1, 76--86. \MR{1217353}

\bibitem[HMS94]{HMS_pic}
Michael~J. Hopkins, Mark Mahowald, and Hal Sadofsky, \emph{Constructions of
  elements in {P}icard groups}, Topology and representation theory ({E}vanston,
  {IL}, 1992), Contemp. Math., vol. 158, Amer. Math. Soc., Providence, RI,
  1994, pp.~89--126. \MR{1263713}

\bibitem[HMS17]{matstohea_piceo}
Drew Heard, Akhil Mathew, and Vesna Stojanoska, \emph{Picard groups of higher
  real {$K$}-theory spectra at height {$p-1$}}, Compos. Math. \textbf{153}
  (2017), no.~9, 1820--1854. \MR{3705278}

\bibitem[HS99]{hovstrmemoir}
Mark Hovey and Neil~P. Strickland, \emph{Morava {$K$}-theories and
  localisation}, Mem. Amer. Math. Soc. \textbf{139} (1999), no.~666, viii+100.
  \MR{1601906}

\bibitem[HS14]{HS-KRD}
Drew Heard and Vesna Stojanoska, \emph{{$K$}-theory, reality, and duality}, J.
  K-Theory \textbf{14} (2014), no.~3, 526--555. \MR{3349325}

\bibitem[MR99]{MahowaldRezk}
Mark Mahowald and Charles Rezk, \emph{Brown-{C}omenetz duality and the {A}dams
  spectral sequence}, Amer. J. Math. \textbf{121} (1999), no.~6, 1153--1177.
  \MR{1719751}

\bibitem[Nav10]{nave1}
Lee~S. Nave, \emph{The {S}mith-{T}oda complex {$V((p+1)/2)$} does not exist},
  Ann. of Math. (2) \textbf{171} (2010), no.~1, 491--509. \MR{2630045}

\bibitem[Rav84]{ravloc}
Douglas~C. Ravenel, \emph{Localization with respect to certain periodic
  homology theories}, Amer. J. Math. \textbf{106} (1984), no.~2, 351--414.
  \MR{737778}

\bibitem[Rav86]{ravgreen}
D.~C. Ravenel, \emph{Complex cobordism and stable homotopy groups of spheres},
  Pure and Applied Mathematics, vol. 121, Academic Press Inc., Orlando, FL,
  1986.

\bibitem[Sto12]{Stoj-th}
Vesna Stojanoska, \emph{Duality for topological modular forms}, Doc. Math.
  \textbf{17} (2012), 271--311. \MR{2946825}

\bibitem[Str00]{StrickGrossHop}
N.~P. Strickland, \emph{Gross-{H}opkins duality}, Topology \textbf{39} (2000),
  no.~5, 1021--1033. \MR{1763961}

\bibitem[Sym04]{symtate}
Peter Symonds, \emph{The {T}ate-{F}arrell cohomology of the {M}orava stabilizer
  group {$S_{p-1}$} with coefficients in {$E_{p-1}$}}, Homotopy theory:
  relations with algebraic geometry, group cohomology, and algebraic
  {$K$}-theory, Contemp. Math., vol. 346, Amer. Math. Soc., Providence, RI,
  2004, pp.~485--492. \MR{2066512}

\end{thebibliography}

\providecommand{\bysame}{\leavevmode\hbox to3em{\hrulefill}\thinspace}
\providecommand{\MR}{\relax\ifhmode\unskip\space\fi MR }
\providecommand{\MRhref}[2]{%
  \href{http://www.ams.org/mathscinet-getitem?mr=#1}{#2}
}
\providecommand{\href}[2]{#2}

\end{document}